\documentclass[english, 11pt]{amsart}

\usepackage{listings}
\usepackage{xcolor}
\usepackage{mathptmx}
\usepackage{tikz}
\usepackage{aliascnt}
\usepackage{mathrsfs}
\usepackage[all,poly,knot]{xy}
\usepackage{hyperref}
\usepackage{comment}  
\usepackage{csquotes}   
\usepackage{amssymb,amsbsy,amsmath,amsfonts,amssymb,amscd,
    graphics,color,footmisc,fancyhdr,multicol,fancybox,
    graphicx,mathrsfs,rotating,ifthen,wasysym}
\usepackage{tikz-cd}
\usepackage{nameref}
\usepackage[cal=cm]{mathalpha}
\usepackage{pdfpages}
\usepackage{multirow}
\usepackage{nccrules}
\usepackage{textcomp}
\usepackage{framed}
\usepackage{graphicx}
\usepackage{pgf,tikz}
\usetikzlibrary{arrows}
\usepackage{lipsum}
\usepackage{mathtools}

\usepackage{geometry}
\makeatletter\newcommand{\leqnomode}{\tagsleft@true}
\newcommand{\reqnomode}{\tagsleft@false}\makeatother

\def\log{\mathrm{log}\,}

\theoremstyle{plain}

\newtheorem{thm}{Theorem}[section]  

\newtheorem{corollary}[thm]{Corollary} 

\newtheorem{lemma}[thm]{Lemma}
\newtheorem{proposition}[thm]{Proposition}
\newtheorem{claim}[thm]{Claim}
\newtheorem*{fact}{Fact}
\newtheorem{question}[thm]{Question}

\newtheorem{theorem}[thm]{Theorem}

\theoremstyle{remark}

\theoremstyle{definition}

\numberwithin{equation}{section}

\makeatletter
\long\def\@makefntext#1{%
  \parindent 0pt
  \noindent
  \hb@xt@0em{\hss\@makefnmark}#1}
\makeatother

\usepackage[hyperpageref]{backref}

\hypersetup{
    colorlinks,
    linkcolor={red!50!black},
    citecolor={blue!62!black},
    urlcolor={blue!80!black}
}

\allowdisplaybreaks[4]

\makeatletter
\newcommand\thankssymb[1]{\textsuperscript{\@fnsymbol{#1}}}
\makeatother
\reqnomode

\usepackage{amsthm}
\makeatletter
\newenvironment{proofof}[1]
  {\proof[\textit{Proof of #1:}]
   \let\@addpunct\@gobble}
  {\endproof}
\makeatother

\title[Hole Phenomenon of Gaussian Analytic Functions with Power-exponential Weights]{Hole Phenomenon of Gaussian Analytic Functions with Power-exponential Weights}
\author{Yun-Heng Du}

\begin{document}
\begin{abstract}
    We establish the \emph{hole phenomenon} for the Gaussian analytic function
    \[
        F_{\beta}(z)=\sum_{n=0}^{\infty}\frac{\xi_{n}}{\sqrt{\Gamma\bigl(\frac{2}{\beta}(n+1)\bigr)}}\,z^{n},
    \]
    associated with the power-exponential weight $e^{-|z|^{\beta}}$ on $\mathbb{C}$, where $\beta>0$.
    Under the condition that $F_{\beta}(z)$ has no zeros in $D(0,r)$, the scaled zero counting measure converges to a limiting measure $\mu_{0}^{\beta}$ vaguely in distribution. This limit exhibits a \emph{forbidden region}
    \[
        \bigl\{1<|z|<e^{1/\beta}\bigr\},
    \]
    which zeros asymptotically avoid.
    This generalizes the remarkable discovery of Ghosh and Nishry for the Gaussian entire function (the case $\beta=2$), who first revealed this striking conditional convergence and the emergence of a hole. Our analysis extends their phenomenon to the entire family of power-exponential weights.
\end{abstract}
\maketitle
\begin{center}
\small
Academy of Mathematics and Systems Science, Chinese Academy of Sciences, Beijing 100190, China\\
duyunheng@amss.ac.cn
\end{center}

\vspace{1em}


\section{Introduction}
Gaussian analytic functions (GAFs) are central objects in random complex analysis (see, e.g., \cite{hkpv} for a comprehensive introduction), with the Gaussian entire function (GEF) serving as a canonical example. Consider a sequence $\{\xi_n\}_{n\in\mathbb{N}}$ of independent and identically distributed standard complex Gaussian random variables. The monomials
\begin{align}
    \label{eq:1.1}
    \bigg\{\frac{1}{\sqrt{\pi}}\frac{z^n}{\sqrt{n!}} \, \bigg\}_{n\in\mathbb{N}}
\end{align}
form an orthonormal system with respect to the Gaussian weight $e^{-|z|^2}$, and the associated Gaussian entire function is defined as the random series
\[
    F(z) = \sum_{n=0}^{\infty} \frac{\xi_n}{\sqrt{n!}} \, z^{\,n},
\]
which converges almost surely to an entire function.

The Gaussian entire function was introduced by Bogomolny, Bohigas, and Leboeuf in \cite{bbl92}. Further study of the GEF has been carried out in \cite{bbl96}, \cite{k93} and \cite{ss}. An extensive literature has since developed; see, for example, \cite{fh, st04, ns}.

Among the various properties of the Gaussian entire function, the hole probability---defined as the probability that a random function has no zeros in a specified region---has received particular attention. Foundational work on this topic includes the contributions of Zeltoni-Zelditch \cite{zz}, Nishry \cite{n} and Sodin-Tsirelson \cite{st05}. More recently, Ghosh-Nishry \cite{gn} discovered a remarkable property of the Gaussian entire function for the disc model, that is, the so called hole phenomenon, which was further developed by Nishry and Wennman for the simply connected model in \cite{nw}. We say a hole phenomenon occurs when, conditioned on the absence of zeros in a domain, a forbidden region emerges outside of it, which is asymptotically avoided by zeros. This paper focuses on the hole phenomenon for certain 
Gaussian analytic functions on $\mathbb{C}$. For paralleling results on compact Riemann surfaces with different probability models, we refer to the work of Dinh-Ghosh-Wu in \cite{dgw} and Wu-Xie in \cite{wx}.

A natural and hitherto open question is whether this striking hole phenomenon is specific to the Gaussian weight ($\beta=2$), or if it is a more universal feature of GAFs associated with power-exponential weights. In this paper, we resolve this question affirmatively for the entire family of the power-exponential weights $e^{-|z|^\beta}$ with $\beta>0$. A direct computation shows that the monomials
\[
\bigg\{\sqrt{\frac{\beta}{2\pi}}\frac{z^n}{\sqrt{\Gamma\big(\frac{2}{\beta}(n+1)\big)}}\bigg\}_{n\in\mathbb{N}}
\]
form an orthonormal basis with respect to the weight $e^{-|z|^\beta}$, which reduces to \eqref{eq:1.1} when $\beta=2$. Thus we consider the Gaussian analytic function
\[
F_\beta(z)\coloneqq\sum_{n=0}^{\infty}\frac{\xi_n}{\sqrt{\Gamma\Bigl(\frac{2}{\beta}(n+1)\Bigr)}}\,z^n.
\]

Our main result establishes that the hole phenomenon holds universally for the entire family $\{F_\beta\}_{\beta>0}$. The following is our main result:

\begin{theorem}
\label{thm:1.1}
    Let $\phi$ be a smooth function on $\mathbb{C}$ with compact support, define
    \[
        n_{F_\beta}(\phi; r) = \sum_{F_\beta(z)=0} \phi\!\left( \frac{z}{r} \right),
        \qquad
        n_{F_\beta}(r) = \#\{\,z \in \mathbb{D}(0,r) : F_\beta(z)=0 \,\}.
    \]
    As $r \to \infty$, one has the asymptotic formula
    \[
        \mathbb{E}\bigl[\, n_{F_\beta}(\phi; r) \,\big|\, n_{F_\beta}(r)=0 \,\bigr]
        = r^{\beta} \int_{\mathbb{C}} \phi(z) \,d\mu_{0}^{\beta}(z)
        + O\!\left( r^{\beta/2} \log^2 r \right).
    \]
    Here,
    \[
        \mu_0^\beta \coloneqq 
        \frac{\beta e}{2} \, m_{\mathbb{S}^1} 
        + \frac{\beta}{2} \, \widehat{m}^{\beta}\big|_{ \{|z| \ge e^{1/\beta}\} },
    \]
    where $m_{\mathbb{S}^1}$ denotes the normalized uniform measure on the unit circle $\mathbb{S}^1 = \{ |z|=1 \}$, and $\widehat{m}^{\beta}$ is a measure on $\mathbb{C}$ whose expression in polar coordinates is
    \[
        d\widehat{m}^{\beta}(z)= \frac{\beta}{2\pi} \, r^{\beta-1} \, dr \, d\theta \qquad(z=r e^{i\theta}).
    \]
\end{theorem}

\begin{corollary}
\label{cor:1.2}
    Let $[\mathcal{Z}_r^{\beta}]$ denote the random zero counting measure of the scaled function $F_{\beta, r}(z) \coloneqq F_\beta(rz)$, conditioned on the event $\{ n_{F_\beta}(r) = 0 \}$. Then, as $r \to \infty$,
    \[
        \frac{1}{r^{\beta}} \,[\mathcal{Z}_r^{\beta}] \;\xrightarrow[]{\,\,}\; \mu_0^{\beta}
    \]
    vaguely in distribution.
\end{corollary}

Observe that the limiting measure $\mu_0^\beta$ charges no mass in the annular region $\{1<|z|<e^{1/\beta}\}$. Therefore, this region is called the \emph{forbidden region}. The geometry of this region depends on $\beta$: the annulus shrinks as $\beta$ increases, collapses to the empty set as $\beta\to\infty$, and expands to the exterior of the closed unit disk $\mathbb{C} \setminus \overline{\mathbb{D}}$ as $\beta\to0^{+}$. The case $\beta=2$, corresponding to the classical Gaussian entire function, was established by Ghosh and Nishry \cite{gn}. The present work extends their result to the whole family of power‑exponential weights $\beta>0$.

A key ingredient in our approach is the minimization of a logarithmic energy functional. Denote by $\mathcal{M}_1(\mathbb{C})$ the set of probability measures on $\mathbb{C}$; we study the minimization of the functional  
$$I_{\alpha,\beta}(\mu)=2\sup_{w\in\mathbb{C}}\bigg(U_{\mu}(w)-\frac{|w|^{\beta}}{\beta\alpha}\bigg)-\Sigma(\mu),\qquad \mu\in\mathcal{M}_1(\mathbb{C}),$$  
where $U_{\mu}(w)=\int_{\mathbb{C}} \log |z-w| \, d\mu(z)$ and $\Sigma(\mu)=\int_{\mathbb{C}} U_{\mu}(w) \, d\mu(w)$. Throughout this paper, whenever we consider minimizing $I_{\alpha,\beta}$, we restrict ourselves to probability measures on $\mathbb{C}$ with compact support and finite logarithmic energy.

To characterize the minimizers of $I_{\alpha,\beta}$ under these constraints, we prove the following:

\begin{proposition}
\label{prop:1.3}
(1) For any $p\in[0,1)$, the minimizer of $I_{\alpha,\beta}(\mu)$ over $\{\mu\in\mathcal{M}_1(\mathbb{C}):\mu(\mathbb{D})\leq p/\alpha\}$ is the measure  
$$\mu_{\alpha,p}^{\beta}=\frac{q-p}{\alpha}m_{\mathbb{S}^1}+\frac{1}{\alpha}\bigg(\widehat{m}^{\beta}\big|_{\{|z|\leq p^{1/\beta}\}}+\widehat{m}^{\beta}\big|_{\{q^{1/\beta}\leq |z|\leq \alpha^{1/\beta}\}}\bigg).$$  
(2) For any $p\in(1,e)$, the minimizer of $I_{\alpha,\beta}(\mu)$ over $\{\mu\in\mathcal{M}_1(\mathbb{C}): \mu(\overline{\mathbb{D}})\geq p/\alpha\}$ is the measure  
$$\mu_{\alpha,p}^{\beta}=\frac{p-q}{\alpha}m_{\mathbb{S}^1}+\frac{1}{\alpha}\bigg(\widehat{m}^{\beta}\big|_{\{|z|\leq q^{1/\beta}\}}+\widehat{m}^{\beta}\big|_{\{p^{1/\beta}\leq |z|\leq \alpha^{1/\beta}\}}\bigg).$$  
(3) For any $p\in[e,\alpha)$, the minimizer of $I_{\alpha,\beta}(\mu)$ over $\{\mu\in\mathcal{M}_1(\mathbb{C}): \mu(\overline{\mathbb{D}})\geq p/\alpha\}$ is the measure  
$$\mu_{\alpha,p}^{\beta}=\frac{p}{\alpha}m_{\mathbb{S}^1}+\frac{1}{\alpha}\widehat{m}^{\beta}\big|_{\{p^{1/\beta}\leq |z|\leq \alpha^{1/\beta}\}}.$$  
\end{proposition}

The proof of Proposition~\ref{prop:1.3} is given in Section~\ref{sect:3} via the method of undetermined coefficients. In particular, the measure  
$$\mu_{\alpha,0}^{\beta}=\frac{e}{\alpha}m_{\mathbb{S}^1}+\frac{1}{\alpha}\widehat{m}^{\beta}\big|_{\{e^{1/\beta}\leq |z|\leq \alpha^{1/\beta}\}}$$  
minimizes $I_{\alpha,\beta}(\mu)$ over the class $\{\mu\in\mathcal{M}_1(\mathbb{C}): \mu(\mathbb{D})=0\}$. Moreover, one has the limit relation
$$\lim_{\alpha\rightarrow\infty}\frac{\beta}{2}\mu_{\alpha,0}^{\beta}=\mu_0^\beta,$$
which connects the variational minimizer with the limiting measure appearing in Theorem~\ref{thm:1.1}.

Another essential tool is the truncation of the infinite series defining $F_\beta$ into a dominant polynomial part plus a negligible tail. This truncation is performed by choosing an appropriate scale $\alpha$ so that the degree of the polynomial part is $N = \frac{\beta\alpha}{2} r^\beta$.

Building on these developments, we obtain the following sharp asymptotic estimate for the \emph{hole probabilities} $\mathbb{P}[\,n_{F_\beta}(r)=0\,]$, which constitutes a key step toward the proof of Theorem~\ref{thm:1.1}.

\begin{proposition}
\label{prop:1.4}
    For any $p \ge 0$ with $p \neq 1$, the following asymptotic estimate holds as $r \to \infty$:
    \[
    \mathbb{P}\bigg[\, n_{F_\beta}(r) = \Big\lfloor \frac{\beta p}{2} r^{\beta} \Big\rfloor \,\bigg]
        = \exp\bigg( -\frac{\beta Z_p}{2} r^{2\beta} + O\big( r^{\beta} \log^{2} r \big) \bigg),
    \]
    where $Z_p$ is defined by
    \[
    Z_p = 
    \begin{cases}
        \displaystyle 
        \frac{e^2}{4}, & p=0, \\[10pt]
        \Bigl| \frac{1}{4} \bigl( q^{2}(2\log q-1) - p^{2}(2\log p-1) \bigr) \Bigr|, & 0 \le p < e,\; p \neq 1, \\[10pt]
        \displaystyle \frac{1}{4} \, p^{2}(2\log p-1), & p \ge e,
    \end{cases}
    \]
    and $q$ is defined as follows: if $p=0$, then $q=e$; if $0<p<e$ and $p\neq1$, then $q$ is the unique positive number different from $p$ that satisfies $q(\log q-1)=p(\log p-1)$; and if $p\ge e$, then $q=0$.
\end{proposition}

With Proposition~\ref{prop:1.4} established, the last main ingredient required for the proof of Theorem~\ref{thm:1.1} is the following estimate.

\begin{lemma}  
\label{lem:1.5}  
Define the event
\[
L_{\beta}(p,\phi,\lambda;r)=\bigg\{ \Big|\, n_{F_\beta}(\phi;r) - r^\beta \int_{\mathbb{C}} \phi \, d\mu_p^\beta \,\Big| \geq \lambda \bigg\},
\]
where
\[
\mu_p^\beta=\lim_{\alpha\to\infty}\frac{\beta\alpha}{2}\,\mu_{\alpha,p}^\beta.
\]

Fix $S>0$. Then for all sufficiently large $r$ and every $\lambda \in (0, Sr^\beta)$, there exist constants $C_p, C_\phi > 0$ such that
$$
\mathbb{P}\!\bigg[\, L_\beta(p,\phi,\lambda;r) \;\Big|\; n_{F_\beta}(r) \leq \frac{\beta p}{2} \, r^\beta \,\bigg]
\leq \exp\!\bigg( -\frac{C_p}{\mathfrak{D}(\phi)}\, \lambda^2 + C_\phi \, r^\beta \log^2 r \bigg).$$
\end{lemma}

Lemma~\ref{lem:1.5} supplies a sharp conditional large‑deviation bound for the linear statistics of $F_\beta$. We will use this estimate to complete the proof of Theorem~\ref{thm:1.1} in Section~\ref{sect:5}. 

The results obtained for the power-exponential weights naturally lead to the question of whether the hole phenomenon extends to Gaussian analytic functions associated with more general weights and to more general domains. The work of Nishry–Wennman \cite{nw} on hole probabilities for simply connected domains, together with the recent advances of Wu–Xie \cite{wx} for domains with smooth boundary on compact Riemann surfaces (see \cite{wx} for the definition of smooth boundary in this context), suggests that the hole phenomenon may be universal under suitable conditions on the weight and the domain. This raises the following natural question:

\begin{question}
Under what conditions on the radial symmetric function $Q: \mathbb{C} \to \mathbb{R}$, does the Gaussian analytic function
\[
F_Q(z) \coloneqq \sum_{n=0}^\infty \frac{\xi_n}{\sqrt{a_n^Q}} \, z^n,
\qquad
a_n^Q \coloneqq \int_0^\infty r^{2n+1} e^{-Q(r)} \, dr,
\]
associated to the weight $e^{-Q(z)}$, exhibit the hole phenomenon for every bounded domain $D \subset \mathbb{C}$ with smooth boundary?
\end{question}

A classical result \cite[Chapter~IV,~Theorem~6.1]{st} provides a variational characterization that is suggestive for this problem: under certain conditions on $Q$, the minimizer of the associated energy functional
\[
I_Q(\mu)\coloneqq 2\sup_{w\in\mathbb{C}}\big(U_{\mu}(w)-Q(w)\big)-\Sigma(\mu),\qquad \mu\in\mathcal{M}_1(\mathbb{C})
\]
has a simple structure, which plays a crucial role in establishing the hole phenomenon for the power-exponential case. It is therefore natural to conjecture that analogous conditions on $Q$ suffice to guarantee the hole phenomenon for all bounded domains with smooth boundary.

\bigskip
\noindent
\textbf{Structure of the paper.} The remainder of the paper is organized as follows. Section~\ref{sect:2} introduces a truncation of the infinite series and proves several lemmas that will be used in the subsequent proofs. Section~\ref{sect:3} establishes an upper bound for the hole probability, in which the characterization of the minimizers of $I_{\alpha,\beta}(\mu)$ (Proposition~\ref{prop:1.3}) plays a key role. A matching lower bound is derived in Section~\ref{sect:4}, thereby completing the proof of Proposition~\ref{prop:1.4}. Section~\ref{sect:5} contains the proof of Lemma~\ref{lem:1.5}, which is then used to finish the proof of Theorem~\ref{thm:1.1} and Corollary~\ref{cor:1.2}.

\section{Preliminaries}
\label{sect:2}

By Stirling's formula, for large integer $k\gg 1$, writing $C_\beta=2^{2/\beta}\beta^{1/2-2/\beta}$, one has
\begin{align}
    \label{eq:2.1}
   C_\beta k^{2/\beta-1/2}\big(\tfrac{2k}{\beta e}\big)^{2k/\beta}\leq \Gamma\big(\tfrac{2}{\beta}(k+1)\big) \leq  2C_\beta k^{2/\beta-1/2}\big(\tfrac{2k}{\beta e}\big)^{2k/\beta}
\end{align}
and consequently
$$\tfrac{k^{1/4-1/\beta}}{\sqrt{2C_\beta}} \big(\tfrac{\beta er^\beta}{2k}\big)^{k/\beta}\leq\tfrac{r^k}{\sqrt{\Gamma\big(\frac{2}{\beta}(k+1)\big)}}\leq \tfrac{k^{1/4-1/\beta}}{\sqrt{C_\beta}}\big(\tfrac{\beta er^\beta}{2k}\big)^{k/\beta}$$
for any $r>0$.
We will need several estimates for the tail term $$ T_N(z) = \sum_{k=N+1}^{\infty}  \frac{\xi_k z^k}{\sqrt{\Gamma\big(\tfrac{2}{\beta}(k+1)\big)}}$$and related quantities, which are collected in the following lemmas.

\begin{lemma}
\label{lem:2.1}
Let $B\geq 1$, and suppose that $N = \tfrac{\beta\alpha}{2} r^\beta$ is a positive integer, with $\alpha\geq (4B)^\beta$ and $r \gg 1$ sufficiently large. Then, outside an event of probability at most $\exp(-Cr^{3\beta})$,
$$ |T_N(z)| \leq \exp\bigg(\frac{N}{\beta}\log\frac{4B^\beta}{\alpha}\bigg) \quad \text{for all} \ |z| \leq Br. $$
\end{lemma}

\begin{proof}
Our proof follows the strategy of \cite[Lemma 3.3]{gn}. We begin by invoking the following lemma:
\begin{lemma}[{\cite[Lemma 3.1]{gn}}]
For any $r > 2$, one has
$$
\mathbb{P}\Bigg[\bigcap_{k=0}^{\infty}\Big\{|\xi_k| \leq \sqrt{r^6 + k}\Big\}\Bigg] \geq 1 - C\exp(-r^6).
$$
\end{lemma}

Applying this lemma with $r$ replaced by $r^{\beta/2}$, we deduce that for any $r > 2^{2/\beta}$, $|\xi_k| \leq \sqrt{r^{3\beta} + k}$ for all $k\in\mathbb{N}$, except on an event whose probability does not exceed $\exp(-C r^{3\beta})$. Denote
\begin{align}
    \label{eq:2.2}
    d_k = \tfrac{r^k \sqrt{r^{3\beta} + k}}{\sqrt{\Gamma\left(\frac{2}{\beta}(k+1)\right)}}.
\end{align}
By Stirling's formula, for all $k \geq N= \frac{\beta\alpha}{2} r^\beta$ with $k\in\mathbb{N}$, one has
\begin{align}
\label{eq:2.3}
\frac{d_{k+1}}{d_k}
&= r\sqrt{1 + \frac{1}{r^{3\beta} + k}} \cdot \sqrt{\tfrac{\Gamma\left(\frac{2}{\beta}(k+1)\right)}{\Gamma\left(\frac{2}{\beta}(k+2)\right)}} \nonumber\\
\text{[use \eqref{eq:2.1}]\qquad} &\leq r\sqrt{2(1 + \frac{1}{r^{3\beta} + k})}\left(1+\frac{1}{k}\right)^{1/4-1/\beta-k/\beta}\left(\frac{\beta e}{2(k+1)}\right)^{1/\beta}\nonumber\\
&<\frac{2}{\alpha^{1/\beta}}\leq\frac{1}{2B}.
\end{align}
since $\alpha\geq (4B)^\beta$. Thus, outside the exceptional event, for any $z \in D(0, Br)$ with $r$ large enough, we obtain:
\begin{align*}
|T_N(z)| &\leq \sum_{k=N+1}^{\infty} d_k B^k\\
\text{[use \eqref{eq:2.3}]\qquad}&\leq 2d_{N+1}B^{N+1}\\
\text{[use \eqref{eq:2.2}]\qquad}&\leq 2\sqrt{r^{3\beta}+ N + 1} \frac{(N+1)^{1/4-1/\beta}}{\sqrt{2C_\beta }}\bigg(\frac{\beta er^\beta B^\beta}{2(N+1)}\bigg)^{(N+1)/\beta}\\
&\leq \exp\bigg(\frac{N}{\beta}\log\frac{4B^\beta}{\alpha}\bigg).
\end{align*}
This finishes the proof.
\end{proof}

\begin{lemma}
\label{lem:2.3}
For sufficiently large $r\gg 1$,
$$  \mathbb{P}\Big[M_{F_\beta}(r)  > \exp(r^\beta)\Big]  \leq  \exp\Big(-\exp\big(\tfrac{1}{4} r^\beta\big)\Big),$$
where $M_{F_\beta}(r)\coloneqq \max\{|F_\beta(z)| : |z| \leq r\}$.
\end{lemma}

\begin{proof}
This proof follows the idea of \cite[Lemma 1]{st05}. Consider the main part
$$ \Sigma_1  =  \sum_{k=0}^{\lfloor\beta er^\beta\rfloor-1} |\xi_k| \frac{r^k}{\sqrt{\Gamma\left(\frac{2}{\beta}(k+1)\right)}}, $$
and the tail
$$ \Sigma_2  =  \sum_{k=\lfloor\beta er^\beta\rfloor}^\infty |\xi_k| \frac{r^k}{\sqrt{\Gamma\left(\frac{2}{\beta}(k+1)\right)}}. $$

For any fixed $\delta\in(0,\tfrac{1}{4}]$, we introduce the event
$$ A_r  =  \bigg\{ |\xi_k| \leq \exp\big(\frac{2\delta r^\beta}{3}\big) \text{ for } 0 \leq k \leq \lfloor\beta er^\beta\rfloor, \quad |\xi_k| \leq 2^{k/3\beta} \text{ for } k > \lfloor\beta er^\beta\rfloor \bigg\}. $$

By \eqref{eq:2.1}, there exists $N_0\in\mathbb{N}_+$ such that for any $k\geq N_0$, we have $\Gamma\big(\frac{2}{\beta}(k+1)\big)\geq C_\beta k^{2/\beta-1/2}\big(\frac{2k}{\beta e}\big)^{2k/\beta}$. Hence, for large $r$ and on the event $A_r$, we obtain:
\begin{align*}
\Sigma_1^2  &\leq  \bigg( \sum_{k=0}^{\lfloor\beta er^\beta\rfloor} |\xi_k|^2 \bigg) \bigg( \sum_{k=0}^{\lfloor\beta er^\beta\rfloor} \frac{r^{2k}}{\Gamma\big(\frac{2}{\beta}(k+1)\big)} \bigg)\\
&\leq (1+\beta er^\beta) e^{4\delta r^\beta/3}\bigg(C_\beta\max\{N_0^{2/\beta-1/2},(\beta er^\beta)^{2/\beta-1/2}\}\sum_{k=N_0}^{\lfloor\beta er^\beta\rfloor}\big(\frac{\beta e}{2k}\big)^{2k/\beta}r^{2k}+Cr^{N_0 -1}\bigg)
\end{align*}
for some constant $C>0$. To bound the items $\big(\frac{\beta e}{2k}\big)^{2k/\beta}r^{2k}$, we set $l=\frac{2k}{\beta}$. Since the function $h(l)\coloneqq\big(\frac{e}{l}\big)^l r^{l\beta}$ attains its maximum at $l=r^\beta$, with maximal value $\exp(r^\beta)$, one has
$$
\big(\frac{\beta e}{2k}\big)^{2k/\beta}r^{2k}\leq \exp(r^\beta),
$$
and consequently
$$
\Sigma_1^2\leq \exp\big((1+\tfrac{5}{3}\delta)r^\beta\big).
$$

For $\Sigma_2$, we obtain
$$
\Sigma_2 \leq \sum_{k=\lfloor\beta er^\beta\rfloor+1}^\infty |\xi_k|\frac{k^{1/4-1/\beta}}{\sqrt{C_\beta }}\bigg(\frac{\beta er^\beta}{2k}\bigg)^{k/\beta}
       \leq \sum_{k=\lfloor\beta er^\beta\rfloor+1}^\infty 2^{-k/2\beta} < 1.
$$

Thus, $M_{F_\beta}(r)\leq\exp\big(\big(\tfrac{1}{2}+\delta\big)r^\beta\big)$ on the event $A_r$. Finally, we establish the following bound for the probability of the complement of $A_r$:
$$
\mathbb{P}(A_r^c) \leq \beta er^\beta \exp\!\Big(-\exp\!\Big(\frac{4\delta r^\beta}{3}\Big)\Big)
                + \sum_{k=\lfloor\beta er^\beta\rfloor+1}^\infty \exp\!\big(-2^{2k/3\beta}\big)
                < \exp\!\big(-\exp(\delta r^\beta)\big),
$$
which completes the estimate.
\end{proof}

\begin{lemma}
\label{lem:2.4}
For sufficiently large $r\gg 1$,
$$ \mathbb{P}\Big[n_{F_\beta}(r) >r^{2\beta}\Big] \leq \exp(-r^{3\beta}). $$
\end{lemma}

\begin{proof}
This proof follows the method of the proof of \cite[Theorem 1]{k06}. For any $R \geq r^2$ and $r\gg 1$, by Jensen's formula,
\begin{align*}
n_{F_\beta}(r) \log \frac{R}{r} &\leq \int_r^R \frac{n_{F_\beta}(u)}{u} du\\
&= \int_0^{2\pi} \log |F_\beta(Re^{i\theta})| \frac{d\theta}{2\pi} - \int_0^{2\pi} \log |F_\beta(re^{i\theta})| \frac{d\theta}{2\pi}.
\end{align*}

By the proof of Lemma~\ref{lem:2.3}, for any small $\epsilon > 0$, one has
$$ \mathbb{P}\Big[\log M_{F_\beta}(t) > \big(\tfrac{1}{2} + \epsilon\big)t^\beta\Big] \leq \exp(-\exp(\epsilon t^\beta)). $$

Assuming $n_{F_\beta}(r)> m\coloneqq R^\beta$ and $\log M_{F_\beta}(R) \leq \big(\tfrac{1}{2} + \epsilon\big)R^\beta$, we deduce
$$ -\int_0^{2\pi} \log |F_\beta(re^{i\theta})| \frac{d\theta}{2\pi} \geq \frac{m}{\beta}\log m - \bigg(\frac{1}{2} + \epsilon+\log r\bigg)m=\frac{m}{\beta}\big(\log m - \beta\log r - \beta(\frac{1}{2}+\epsilon)\big). $$

Thus, we bound the probability as
\begin{align}
\label{eq:2.4}
\mathbb{P}\Big[n_{F_\beta}(r)> m\Big] &\leq  \mathbb{P}\Big[\log M_{F_\beta}(R) >\big(\tfrac{1}{2} + \epsilon\big)m\Big]\nonumber\\
&\qquad+ \mathbb{P}\bigg[-\int_0^{2\pi} \log |F_\beta (re^{i\theta})| \frac{d\theta}{2\pi} \geq \frac{m}{\beta}\big(\log m - \beta\log r - \beta(\frac{1}{2}+\epsilon)\big)\bigg] \nonumber \\
&\leq \exp\big(-\exp(\epsilon m)\big)\nonumber\\
&\qquad+ \mathbb{P}\left[-\int_0^{2\pi} \log |F_\beta (re^{i\theta})| \frac{d\theta}{2\pi} \geq \frac{m}{\beta}\big(\log m - \beta\log r - \beta(\frac{1}{2}+\epsilon)\big)\right].
\end{align}

To complete the estimate, we invoke the next inequality, which is in the proof of \cite[Lemma 7]{k06}.

\begin{lemma}
For any $\epsilon>0$ and $u>0$, the following inequality holds:
$$
\mathbb{P}\left[-\int_0^{2\pi} \log |F_\beta(re^{i\theta})| \frac{d\theta}{2\pi} \geq u\right] \leq \exp\big(-\exp(C_\epsilon u)\big)+\mathbb{P}\left[\log M_{F_\beta}(\epsilon) \leq -B_{\epsilon}u + \sqrt{u}\right],
$$
where $C_\epsilon,\,B_\epsilon>0$ are constants depend on $\epsilon$, and $B_\epsilon\to 1$ as $\epsilon\to 0^+$.
\end{lemma}

Fix $\epsilon_0>0$ such that $B_{\epsilon_0}\in\big[\frac{1}{2},2\big]$. Since
$$ |\xi_n| = \bigg|\tfrac{\sqrt{\Gamma\big(\frac{2}{\beta}(n+1)\big)}}{2\pi i} \int_{\{|\zeta|=\epsilon_0\}} \frac{F_\beta(\zeta)}{\zeta^{n+1}} d\zeta\bigg| \leq \frac{M_{F_\beta}(\epsilon_0) }{\epsilon_0^n}\sqrt{\Gamma\big(\tfrac{2}{\beta}(n+1)\big)}, $$
it follows that
\begin{align*}
\mathbb{P}\Big[M_{F_\beta}(\epsilon_0)\leq e^{-s}\Big]&\leq\prod_{n=0}^{\infty} \mathbb{P}\bigg[|\xi_n|\leq \frac{e^{-s}}{\epsilon_0^n}\sqrt{\Gamma\big(\tfrac{2}{\beta}(n+1)\big)}\bigg] \\
&\leq \prod_{n=0}^\infty\frac{e^{-2s} }{\epsilon_0^{2n}}\Gamma\big(\tfrac{2}{\beta}(n+1)\big)\\
\text{[by \eqref{eq:2.1}]\qquad} &=\prod_{n=1}^\infty\exp\bigg(-2s-2n\log \epsilon_0+\frac{2n}{\beta}\log\frac{2n}{\beta e}+O(\log n)\bigg)
\end{align*}
for any $s>0$. Restricting the product to indices $n$ from $0$ to $k := \lceil\frac{\beta s}{\log s}\rceil$, we obtain
\begin{align*}
\mathbb{P}\left[M_{F_\beta}(\epsilon_0) \leq e^{-s}\right] &\leq \exp\bigg(-2sk + \frac{k^2}{\beta}\log k+ O(k^2)\bigg)\nonumber\\
&\leq \exp\bigg(-\frac{\beta s^2}{\log s} + O\bigg(\frac{s^2}{(\log s)^2}\bigg)\bigg).
\end{align*}

Therefore,
\begin{align}
\label{eq:2.5}    \mathbb{P}\bigg[-\int_0^{2\pi} \log |F_\beta(re^{i\theta})| \frac{d\theta}{2\pi} \geq u\bigg] \leq \exp\bigg(-\frac{\beta u^2}{\log u} \big(1 + o(1)\big)\bigg).
\end{align}

Combining \eqref{eq:2.4} and \eqref{eq:2.5}, we conclude that
$$\mathbb{P}\Big[n_{F_\beta}(r) > r^{2\beta}\Big] \leq \exp(-r^{3\beta}). $$
\end{proof}

\begin{lemma}
\label{lem:2.6}
There exists a constant $c_\beta>0$ such that, for all $x>0$ and all sufficiently large $r$,\[    \mathbb{P}\left[M_{F_\beta}(r)< e^{-x}\right]    \leq \exp(-c_\beta x r^\beta).\]
\end{lemma}

\begin{proof}
Set
\[
    \eta:=\frac{\beta e}{16},
\]
and
\[
    I_r:=\left\{k\in\mathbb N:\ \eta r^\beta\leq k\leq 2\eta r^\beta\right\}.
\]

By \eqref{eq:2.1}, for every integer
$k\in I_r$,
$$\tfrac{r^k}{\sqrt{\Gamma\big(\frac{2}{\beta}(k+1)\big)}}\geq (2C_\beta)^{-1/2} k^{1/4-1/\beta} \left(\tfrac{\beta e r^\beta}{2k}\right)^{k/\beta}\geq
    (2C_\beta)^{-1/2} k^{1/4-1/\beta}4^{k/\beta}$$
when $r$ is large enough.

Since
$k\in[\eta r^\beta,2\eta r^\beta]$, there exists a constant $C_1=C_1(\beta)>0$ such that,
for all sufficiently large $r$,
\[
    (2C_\beta)^{-1/2}k^{1/4-1/\beta}\geq r^{-C_1},
    \qquad k\in I_r.
\]

Hence, for every $k\in I_r$, one has
\[
    \tfrac{r^k}{\sqrt{\Gamma\big(\frac{2}{\beta}(k+1)\big)}} \geq
    r^{-C_1}\exp\left(\frac{k}{\beta}\log 4\right)\geq
    r^{-C_1}\exp\left(\frac{\eta\log 4}{\beta}r^\beta\right)\geq \exp(C_2 r^\beta),
\]
where $$C_2\coloneqq\frac{\eta\log 4}{2\beta}.$$

Now assume that $M_{F_\beta}(r)< e^{-x}$. For every $k\geq0$,
$$ |\xi_k| = \bigg|\tfrac{\sqrt{\Gamma\big(\frac{2}{\beta}(k+1)\big)}}{2\pi i} \int_{\{|\zeta|=r\}} \frac{F_\beta(\zeta)}{\zeta^{k+1}} d\zeta\bigg| \leq \frac{e^{-x}}{r^k}\sqrt{\Gamma\big(\tfrac{2}{\beta}(k+1)\big)}.$$

Therefore, for every $k\in I_r$, we obtain
\[
    |\xi_k|
    \leq
    e^{-x}\exp(-C_2 r^\beta)
    =
    \exp(-x-C_2 r^\beta).
\]
Thus
$$\{M_{F_\beta}(r)< e^{-x}\}
    \subset
    \bigcap_{k\in I_r}
    \left\{|\xi_k|\leq \exp(-x-C_2 r^\beta)\right\}.$$
By the independence of
$\{\xi_k\}_{k\in I_r}$, one has
\begin{align*}
    \mathbb{P}\left[M_{F_\beta}(r)< e^{-x}\right]
    &\leq
    \prod_{k\in I_r}
    \mathbb P\left\{|\xi_k|\leq \exp(-x-C_2 r^\beta)\right\} \\
    &\leq
    \exp\left(-(2x+2C_2 r^\beta)\# I_r\right) \\
    &\leq
    \exp\left(-\eta x r^\beta\right).
\end{align*}
The assertion follows with $c_\beta=\eta$.
\end{proof}

We will also use the following theorem, which was first presented in \cite[Theorem 4]{r} and later restated in \cite[Theorem~3.7]{gn}:
\begin{theorem}
\label{thm:2.7}
Let $f$ be an entire function and $z_0 \in\mathbb{C}\backslash\{0\}$ is a point such that $|f(z_0)| \geq 1$. Assume that $$E_{\gamma}(r) = D(0, r) \setminus \bigg(\bigcup_{k=1}^m D(\omega_k, \gamma)\bigg)\neq \emptyset,$$
where $r\in(0,\frac{|z_0|}{2}]$, $\gamma\in(0,\frac{1}{4}]$ and $\omega_1, \ldots, \omega_m$ are zeros of $f$ in $D(0, r)$, each repeated according to its multiplicity (so that a zero of multiplicity $k$ appears $k$ times in the list). Then
$$ m_f(r;\gamma):=\min_{z\in E_\gamma(r)}|f(z)| \geq \exp\left(-C\log M_f(3|z_0|)\log\frac{1}{\gamma}\right). $$
\end{theorem}
With the above preparations, we can prove the following lemmas:
\begin{lemma}
\label{lem:2.8}
Let $\rho\gg 1$ be sufficiently large and $N = \tfrac{\beta\alpha}{2}\rho^\beta$, $\alpha\geq (4B)^\beta$, $B\geq 1$, $A,H>0$. Then there exists an event $E_{reg}$ satisfying:

(1) The probability of the exceptional event $E_{reg}^c$ is bounded by
$$  \mathbb{P}(E_{reg}^c) \leq \exp(-CAB^{2\beta}\rho^{2\beta}) + \exp(-2H\rho^{2\beta}) + \exp(-\alpha\rho^{2\beta})$$
for some constant $C>0$.

(2) On $E_{reg}$, one has:
\begin{flalign*}
&\text{$\qquad$ (i) }|T_N(z)| \leq \exp\left(\tfrac{N}{\beta}\log\tfrac{4B^\beta}{\alpha}\right) \text{ for any $|z|\leq B\rho$} &&\text{(see Lemma~\ref{lem:2.1})}& \\[6pt]
&\text{$\qquad$ (ii) }M_{F_\beta}(12B\rho) \leq \exp(12^\beta B^\beta\rho^\beta) &&\text{(see Lemma~\ref{lem:2.3})}& \\[6pt]
&\text{$\qquad$ (iii) }n_{F_\beta}(B\rho) \leq B^{2\beta}\rho^{2\beta} &&\text{(see Lemma~\ref{lem:2.4})}& \\[6pt]
&\text{$\qquad$ (iv) }M_{F_\beta}(4B\rho) \geq \exp(-AB^\beta\rho^\beta) &&\text{(see Lemma~\ref{lem:2.6})}&
\end{flalign*}

(3) $|\xi_N|>\exp(-H\rho^{2\beta})$ and $\sum_{k=0}^N |\xi_k|^2 \leq 2\alpha\rho^{2\beta}$ on $E_{reg}$.

\end{lemma}

\begin{proof}
Applying Lemmas~\ref{lem:2.1}, \ref{lem:2.3}, \ref{lem:2.4}, and \ref{lem:2.6}, we obtain the bounds in (2). The largest probability among the exceptional events is that of (iv), which is bounded above by $\exp(-CAB^{2\beta}\rho^{2\beta})$ for some constant $C>0$.

By Lemma 3.2 in \cite{gn},
$$  \mathbb{P}\left[\sum_{k=0}^N |\xi_k|^2 > 2\alpha\rho^{2\beta}\right] \leq \exp(-\alpha\rho^{2\beta}).  $$

Since $\mathbb{P}\left[|\xi_N|\leq\exp(-H\rho^{2\beta})\right]\leq\exp(-2H\rho^{2\beta})$, (3) follows.
\end{proof}

\begin{lemma}
\label{lem:2.9}
Fix $B\geq 1$. Let $\rho\gg 1$ be sufficiently large, $N=\tfrac{\beta\alpha}{2}\rho^\beta$ with $\alpha\in[\log\rho,2\log\rho]$, $M_0 = (2B)^{2\beta}\rho^{2\beta}$ and $\gamma=\rho^{-s}$, where $s>1+4\beta$ is a constant. Then the following statements hold:

(i) On the event $E_{reg}$, one has
$$m_{F_\beta}(2B\rho;\gamma) > M_{T_N}(2B\rho) \quad\text{and}\quad \frac{1}{|\xi_N|^2} \sum_{k=0}^{N} |\xi_k|^2 < e^{3H\rho^{2\beta}}.$$

(ii) One has the two‑sided bound
$$n_{P_N}(\rho - 2M_0\gamma) \leq n_{F_\beta}(\rho) \leq n_{P_N}(\rho + 2M_0\gamma).$$

(iii) For any smooth function $\phi$ compactly supported in $D(0,B)$, one has
$$|n_{F_\beta}(\phi; \rho) - n_{P_N}(\phi; \rho)| \leq CM_0 \, \omega \big(\phi;\tfrac{2M_0\gamma}{\rho}\big),$$
where $\omega(\phi;t) := \sup\{|\phi(x)-\phi(y)|: x,y\in\mathbb{C},\,|x-y|\le t\}$ denotes the modulus of continuity of $\phi$.
\end{lemma}

\begin{proof}
We follow the approach of the proof of \cite[Lemma~3.11]{gn}. First, (3) of Lemma~\ref{lem:2.8} implies that $$\frac{1}{|\xi_N|^2} \sum_{k=0}^{N} |\xi_k|^2 <e^{3H\rho^{2\beta}}$$ on $E_{reg}$. Applying the maximum modulus principle, there exists a point $z_0\in\mathbb{C}$ with $|z_0| = 4B\rho$ satisfying
$$
|F_\beta(z_0)| \geq \exp(-AB^\beta\rho^\beta).
$$

Set
$$  \widehat{F}_\beta(z) = \frac{F_\beta(z)}{F_\beta(z_0)},  $$
and denote $\widehat{M}_\beta(r) = \max\{|\widehat{F}_\beta(z)| : |z| \leq r\}$. Using condition (ii) of (2) from Lemma~\ref{lem:2.8} and applying Theorem~\ref{thm:2.7} to $\widehat{F}_{\beta}$, we obtain
$$  m_{F_\beta}(2B\rho; \gamma) \geq \exp\left(-C B^\beta\rho^\beta\log\rho\right).  $$

For $|z|\leq 2B\rho$, using Lemma~\ref{lem:2.1}, one has
$$ |T_N(z)| \leq \exp\left(\frac{N}{\beta}\log \frac{4(2B)^\beta}{\alpha}\right) \leq \exp\left(-C\rho^\beta(\log\rho)\log(\log\rho)\right), $$
so $m_{F_\beta}(2B\rho;\gamma) > M_{T_N}(2B\rho)$ for large $\rho$, which proves (i). As for (ii) and (iii), we only need to apply the following lemma to $F_\beta$ and $-T_N$:
\begin{lemma}[{\cite[Lemma 3.8]{gn}}]
Let $f$ and $g$ be entire functions, and $B, \rho \geq 1$. Suppose that $f$ has at most $M > 0$ zeros in $D(0, 2B\rho)$ and $M_g(2B\rho) < m_f(2B\rho; \gamma)$, where $0 < \gamma < \frac{\rho}{2M}$. Then one has
$$n_{f+g}(\rho' - 2M\gamma) \leq n_f(\rho') \leq n_{f+g}(\rho' + 2M\gamma) \quad \forall \rho' \in (2M\gamma, 2B\rho - 2M\gamma)$$
and
$$
|n_f(\phi; \rho) - n_{f+g}(\phi; \rho)| \leq M \cdot \omega\big(\phi; \frac{2M\gamma}{\rho}\big),
$$
where $\phi$ is a smooth function compactly supported in $D(0, B)$. 
\end{lemma}
Thus we finish the proof.
\end{proof}

\section{The Upper Bounds of the Hole Probabilities}
\label{sect:3}
In this section we establish the upper bounds of the hole probabilities in Proposition~\ref{prop:1.4}. Specifically, we prove part (1) for $p\in[0,1)$:
\[
\mathbb{P}\bigg[n_{F_\beta}(r) \leq\big\lfloor\frac{\beta p}{2}r^\beta\big\rfloor\bigg]\leq\exp\bigg(-\frac{\beta Z_p}{2}r^{2\beta}+O\big(r^\beta\log^2 r\big)\bigg).
\]
The upper bounds for parts (2) and (3) of Proposition~\ref{prop:1.3}, which correspond to the case $p>1$, are given by
\[
\mathbb{P}\bigg[n_{F_\beta}(r) \geq\big\lfloor\frac{\beta p}{2}r^\beta\big\rfloor\bigg]\leq\exp\bigg(-\frac{\beta Z_p}{2}r^{2\beta}+O\big(r^\beta\log^2 r\big)\bigg),
\]
and follow from the same argument used for part (1); therefore we will omit the details for brevity. The main strategy is to truncate the series $F_\beta$, relate its zero count to that of a random polynomial $P_{N,L}$, and then compute the hole probabilities via minimizing a logarithmic energy functional.

Let $\phi \in C_c^{\infty}(\mathbb{C})$ be a test function with compact supports in $D(0,B)$, where $B \geq 1$ if fixed. Applying Lemma~\ref{lem:2.8} and \ref{lem:2.9} with the data $\rho = r\gg 1$, $N=\tfrac{\beta\alpha}{2}r^\beta$, $\alpha\in[\log r,2\log r]$, $M_0=(2B)^{2\beta}r^{2\beta}$, $t=\gamma=r^{-s}$, $s>1+4\beta$, $K_0\coloneqq 2M_0\gamma=2(2B)^{2\beta} r^{2\beta-s}$ and $L = \tfrac{r-K_0}{1+t} = r + o(1)$, one has
\begin{align}
n_{P_N}(r - K_0) &\leq n_{F_\beta}(r) \leq n_{P_N}(r + K_0); \\
|n_{F_\beta}(\phi; r) - n_{P_N}(\phi; r)| &\leq CM_0 \, \omega\left(\phi; 2(2B)^{2\beta} r^{2\beta-s-1}\right);\\
\label{eq:3.3}
\frac{1}{|\xi_N|^2} \sum_{k=0}^{N} |\xi_k|^2 &< e^{3H\rho^{2\beta}}.
\end{align}
on $E_{reg}$. First we consider the distribution of the zeroes of the random polynomial $P_{N,L}$:
\begin{lemma}
\label{lem:3.1}
Let $\underline{z} = (z_1, \ldots, z_N)$ denote the zeros of the polynomial
$$P_{N,L}(z) = \sum_{k=0}^N  \frac{\xi_k(Lz)^k}{\sqrt{\Gamma\left(\frac{2}{\beta}(k+1)\right)}}$$
arranged in the uniform random order, that is, the labeling $(z_1,\dots,z_N)$ is chosen uniformly at random from all $N!$ possible permutations of the zero set. The joint distribution of $\underline{z}$ with respect to the Lebesgue measure $m$ is given by
$$f_\beta(\underline{z}) = A_{N,L} |\Delta(\underline{z})|^2 \left( \int_{\mathbb{C}} \prod_{j=1}^N |w - z_j|^2 \, d\nu_L^\beta(w) \right)^{-(N+1)},$$
where
\begin{align*}
A_{N,L} &= \frac{N!\prod_{k=0}^N\Gamma\left(\frac{2}{\beta}(k+1)\right)}{\pi^N \Gamma\left(\frac{2}{\beta}\right)^{N+1}L^{N(N+1)}} \\
&=\exp\bigg(\frac{\beta\alpha^2 r^{2\beta}}{4}(\log\alpha-\frac{3}{2})+O(r^\beta\log^2 r)\bigg), \\
|\Delta(\underline{z})|^2 &= \prod_{j \neq k} |z_j - z_k|, \\
d\nu_L^\beta(w) &= \frac{\beta L^2}{2\pi\Gamma\left(\frac{2}{\beta}\right)} e^{-L^\beta |w|^\beta} \, dm(w).
\end{align*}
\end{lemma}

\begin{proof}
Since $\{\xi_i\}_{0\leq i\leq N}$ are i.i.d. complex Gaussian random variables, their joint density is
$$g(\underline{\xi})=\frac{1}{\pi^{N+1}}\exp\left(-\sum_{k=0}^N|\xi_k|^2\right).$$
Set
$$q_{\underline{z}}(z)\coloneqq P_{N,L}(z)\tfrac{\sqrt{\Gamma\left(\frac{2}{\beta}(N+1)\right)}}{\xi_N L^N}  =\prod_{j=1}^N (z - z_j)=z^N+b_{N-1}z^{N-1}+\cdots+b_1 z+b_0,$$
where
$$b_k=\frac{\xi_k}{\xi_N}\tfrac{\sqrt{\Gamma\left(\frac{2}{\beta}(N+1)\right)}}{\sqrt{\Gamma\left(\frac{2}{\beta}(k+1)\right)}}L^{k-N}.$$
Denote $\underline{b}=(b_0,\ldots,b_{N-1})$, the Jacobian of $T_1:\underline{z}\to\underline{b}$ is $|\Delta(\underline{z})|^2$, and the Jacobian of $$T_2:\underline{\xi}'=(\xi_0,\ldots,\xi_{N-1})\to\underline{b}$$ is
$$\prod_{k=0}^{N-1}\tfrac{\Gamma\left(\frac{2}{\beta}(N+1)\right)}{|\xi_N|^2\Gamma\left(\frac{2}{\beta}(k+1)\right)L^{2(N-k)}}=\tfrac{\Gamma\left(\frac{2}{\beta}(N+1)\right)^{N+1}}{|\xi_N|^{2N}\prod_{k=0}^N\Gamma\left(\frac{2}{\beta}(k+1)\right) L^{N(N+1)}}.$$
Therefore, the joint density of $(\underline{z},\xi_N)$ is
$$g'(\underline{z},\xi_N)=\tfrac{|\xi_N|^{2N}\prod_{k=0}^N\Gamma\left(\frac{2}{\beta}(k+1)\right)L^{N(N+1)}}{\pi^{N+1}\Gamma\left(\frac{2}{\beta}(N+1)\right)^{N+1}}|\Delta(\underline{z})|^2\exp\left(-\sum_{k=0}^N|\xi_k|^2\right).$$

Since
\begin{align}
\label{eq:3.4}
    \int_{\mathbb{C}}|w|^{2k}d\nu_L^\beta(w)=\tfrac{\Gamma\left(\frac{2}{\beta}(k+1)\right)}{\Gamma\left(\frac{2}{\beta}\right)L^{2k}},
\end{align}
for any $L>0$, one has
$$\int_{\mathbb{C}}|P_{N,L}(w)|^{2}d\nu_L^\beta(w)=\tfrac{1}{\Gamma\left(\frac{2}{\beta}\right)}\sum_{k=0}^N|\xi_k|^2$$
and
\begin{align}
\label{eq:3.5}
    \int_{\mathbb{C}}|q_{\underline{z}}(w)|^{2}d\nu_L^\beta(w)=\tfrac{\Gamma\big(\frac{2}{\beta}(N+1)\big)\sum_{k=0}^N|\xi_k|^2}{\Gamma\big(\frac{2}{\beta}\big)L^{2N}|\xi_N|^2}.
\end{align}
Thus
\begin{align*}
g'(\underline{z},\xi_N) = &\tfrac{|\xi_N|^{2N}\prod_{k=0}^N\Gamma\left(\frac{2}{\beta}(k+1)\right)L^{N(N+1)}}{\pi^{N+1}\Gamma\left(\frac{2}{\beta}(N+1)\right)^{N+1}}|\Delta(\underline{z})|^2\\
&\times \exp\left(-|\xi_N|^2\tfrac{\Gamma\left(\frac{2}{\beta}\right)L^{2N}}{\Gamma\left(\frac{2}{\beta}(N+1)\right)}\int_{\mathbb{C}}|q_{\underline{z}}(w)|^2d\nu_L^\beta(w)\right).
\end{align*}
Integrating $\xi_N$ over $\mathbb{C}$ yields
$$f_\beta(\underline{z})=A_{N,L} |\Delta(\underline{z})|^2 \bigg( \int_{\mathbb{C}} \prod_{j=1}^N |w - z_j|^2 \, d\nu_L^\beta(w) \bigg)^{-(N+1)},$$
where
\begin{align*}
A_{N,L} &= \frac{N! \prod_{k=0}^N \Gamma\left(\frac{2}{\beta}(k+1)\right)}{\pi^N \Gamma\left(\frac{2}{\beta}\right)^{N+1}L^{N(N+1)}}\\
\text{[use \eqref{eq:2.1}]\qquad}&=\exp\bigg(\sum_{k=1}^N\frac{2k}{\beta}\log\frac{2k}{\beta e}-(N^2+N)\log L+O(N\log N)\bigg)\\
&=\exp\bigg(\frac{2}{\beta}(N^2\sum_{k=1}^N\frac{1}{N}\frac{k}{N}\log\frac{k}{N}+\sum_{k=1}^N k\log N)+\frac{\log 2-1-\log\beta}{\beta}N^2\\
&\qquad -(N^2+N)\log L+O(N\log N)\bigg)\\
&=\exp\bigg(\frac{2}{\beta}(N^2\int_0^1 x\log xdx+\frac{N^2}{2}\log N)+\frac{\log 2-1-\log\beta}{\beta}N^2\\
&\qquad -(N^2+N)\log L+O(N\log N)\bigg)\\
&=\exp\bigg(\frac{N^2}{\beta}(\log 2-\frac{3}{2}-\log\beta+\log\frac{N}{L^\beta})-N\log L+O(N\log N)\bigg)\\
&=\exp\bigg(\frac{\beta\alpha^2 r^{2\beta}}{4}(\log\alpha-\frac{3}{2})+O(r^\beta\log^2 r)\bigg).
\end{align*}
This finishes the proof.
\end{proof}

To bound the hole probabilities, we establish the following lemma:
\begin{lemma}
\label{lem:3.2}
Set
$$A_\beta(\underline{z}) \coloneqq\sup_{w\in\mathbb{C}} \Big(|q_{\underline{z}}(w)|^2e^{-L^\beta|w|^\beta}\Big)$$
and
$$S_\beta(\underline{z}) \coloneqq\int_{\mathbb{C}} |q_{\underline{z}}(w)|^2 \, d\nu_L^\beta(w).$$
One has the following properties for $A_\beta(\underline{z})$ and $S_\beta(\underline{z}):$

(i)$\,S_\beta(\underline{z})\geq N^{-3/2} A_\beta(\underline{z}) \geq N^{-3/2}\prod_{j=1}^N \max^2\{1, |z_j|\} \, e^{-L^\beta};$

(ii)$\,\int_{\mathbb{C}^N} A_\beta(\underline{z})^{-b} \, dm(\underline{z}) \leq e^{b L^\beta} \big( \tfrac{\pi b}{b-1} \big)^N$ for any $b > 1;$

(iii)$\,S_\beta(\underline{z}) \leq \tfrac{\Gamma\big(\frac{2}{\beta}(N+1)\big)}{\Gamma\big(\frac{2}{\beta}\big)L^{2N}}e^{3H\rho^{2\beta}}\text{ on $E_{reg}$}.$
\end{lemma}
\begin{proof}
(i) We begin by introducing the function $$\Pi_{N,L}(w,z)=\sum_{k=0}^N \frac{\Gamma\big(\frac{2}{\beta}\big)L^{2k}}{\Gamma\big(\frac{2}{\beta}(k+1)\big)}(w\Bar{z})^k,\,w,\,z\in\mathbb{C}.$$
By \eqref{eq:3.4}, one can estimate
\begin{align*}
    |q_{\underline{z}}(w)|^2 &=\bigg|\int_{\mathbb{C}}\Pi_{N,L}(w,z)q_{\underline{z}}(z)d\nu_L^\beta(z)\bigg|^2\\
    &\leq \bigg(\int_{\mathbb{C}}|q_{\underline{z}}(z)|^2d\nu_L^\beta(z)\bigg)\bigg(\int_{\mathbb{C}}|\Pi_{N,L}(w,z)|^2 d\nu_L^\beta(z)\bigg)\\
   &=S_\beta(\underline{z})\sum_{k=0}^N\frac{L^{2k}|w|^{2k}}{\Gamma\big(\frac{2}{\beta}(k+1)\big)}.
\end{align*}

Set $\widehat{h}(x):=x^{2k}e^{-x^\beta}$, $x\geq 0$. The function $\widehat{h}$ attains its maximum at $x=\big(\frac{2k}{\beta}\big)^{1/\beta}$, with maximal value $\big(\frac{2k}{\beta e}\big)^{2k/\beta}$. Consequently,
\[
\frac{x^{2k}}{\Gamma\big(\frac{2}{\beta}(k+1)\big)e^{x^\beta}}\leq\frac{1}{C_\beta}k^{1/2-2/\beta}\leq\frac{1}{C_\beta}\max\big\{N^{1/2-2/\beta},N_0^{1/2-2/\beta}\big\}
\]
for every integer $k$ with $N_0\leq k\leq N$, where $N_0$ is chosen so that $\Gamma\big(\frac{2}{\beta}(k+1)\big)\geq C_\beta k^{2/\beta-1/2}\big(\frac{2k}{\beta e}\big)^{2k/\beta}$ holds for all $k\geq N_0$.

Let $x=L|w|$, we obtain
\[
|q_{\underline{z}}(w)|^2\leq S_\beta(\underline{z})\sum_{k=0}^N\frac{L^{2k}|w|^{2k}}{\Gamma\big(\frac{2}{\beta}(k+1)\big)}
      \leq N^{3/2}e^{L^\beta|w|^\beta}S_\beta(\underline{z}).
\]
Taking the supremum over $w\in\mathbb{C}$ yields $S_\beta(\underline{z})\geq N^{-3/2}A_\beta(\underline{z})$. Note that
\begin{align*}
    \log A_\beta(\underline{z})&=\sup_{w \in \mathbb{C}} \left(2\log |q_{\underline{z}}(w)|-L^\beta|w|^\beta\right)\\
    &\geq \frac{1}{\pi}\int_0^{2\pi}\log |q_{\underline{z}}(e^{i\theta})|d\theta-L^\beta\\
    &=\frac{1}{\pi}\int_0^{2\pi}\sum_{j=1}^N\log |e^{i\theta}-z_j|d\theta-L^\beta\\
    &=\sum_{j=1}^N 2\log \max\{1, |z_j|\}-L^\beta,
\end{align*}
then $S_\beta(\underline{z})\geq N^{-3/2}A_\beta(\underline{z}) \geq N^{-3/2}\prod_{j=1}^N \max^2\{1, |z_j|\} \, e^{-L^\beta}$.

(ii) By (i), we obtain that
$$\int_{\mathbb{C}^N} A_\beta(\underline{z})^{-b} \, dm(\underline{z}) \leq e^{b L^\beta}\int_{\mathbb{C}^N}\prod_{j=1}^N\big(\max\{1, |z_j|\}\big)^{-2b}dm(\underline{z})=e^{b L^\beta}\big(\tfrac{\pi b}{b-1} \big)^N.$$

(iii) By \eqref{eq:3.5} and \eqref{eq:3.3},
$$S_\beta(\underline{z})\leq\tfrac{\Gamma\big(\frac{2}{\beta}(N+1)\big)\sum_{k=0}^N|\xi_k|^2}{\Gamma\big(\frac{2}{\beta}\big)L^{2N}|\xi_N|^2}\leq \tfrac{\Gamma\big(\frac{2}{\beta}(N+1)\big)}{\Gamma\big(\frac{2}{\beta}\big)L^{2N}}e^{3H\rho^{2\beta}}$$
on $E_{reg}$, which completes our proof.
\end{proof}
Consider the function
\begin{align*}
I_\beta^*(\underline{z}) &= 2\sup_{w \in \mathbb{C}} \bigg( \frac{1}{N}\log |q_{\underline{z}}(w)| - \frac{L^\beta}{2N} |w|^\beta \bigg) - \frac{1}{N^2}\sum_{j \neq k} \log |z_j - z_k| \\
&= \frac{1}{N}\log A_\beta(\underline{z}) - \frac{1}{N^2}\sum_{j \neq k} \log |z_j - z_k|.
\end{align*}
By (i) of Lemma~\ref{lem:3.2},
$$S_\beta(\underline{z})^{-(N+1)} \leq N^{3(N+1+1/N)/2}A_\beta(\underline{z})^{-(1+1/N)}S_\beta(\underline{z})^{1/N}A_\beta(\underline{z})^{-N}$$
on $E_{reg}$. For any subset $Z \subset \mathbb{C}^N$, introduce the set
$$E = Z \cap \bigg\{\underline {z}\in\mathbb C^N:\,S_\beta(\underline z)\le\frac{\Gamma\big(\tfrac{2}{\beta}(N+1)\big)}{\Gamma\big(\tfrac{2}{\beta}\big)L^{2N}}e^{3Hr^{2\beta}}\bigg\}.$$
Applying Lemma~\ref{lem:3.2} with $b=1+\tfrac{1}{N}$ and noting that $L=r+o(1),\,N=\tfrac{\beta\alpha}{2}r^\beta$, we obtain
\begin{align*}
\mathbb{P}[\{\underline{z}\in Z\} \cap E_{reg}] &\leq N^{3(N+1+1/N)/2}A_{N,L} \exp\big(- 2\log L + \tfrac{2}{\beta}\log N + O(1)\big) \\
&\quad \times \int_E \exp\bigg(\sum_{j \neq k} \log |z_j - z_k| - N\log A_\beta(\underline{z})\bigg) A_\beta(\underline{z})^{-(1+1/N)} dm(\underline{z})\\
\text{[by Lemma~\ref{lem:3.1}]\qquad}&=\exp\bigg(\frac{\beta\alpha^2 r^{2\beta}}{4}\big(\log\alpha-\frac{3}{2}-\beta\inf_{\underline{z}\in E} I_\beta^*(\underline{z})\big)+O(r^\beta\log^2 r)\bigg).
\end{align*}

To bound the hole probabilities, we establish a comparison between $I^*(\underline{z})$ and $I_{\alpha,\beta}(\mu_{\underline{z}}^t)$ for sufficiently small $t>0$, where $\mu_{\underline{z}}^t(z) = \tfrac{1}{N}\sum_{j=1}^Nm_{\{|z-z_j|=t\}}$ and $m_{\{|z-z_j|=t\}}$ is the normalized uniform distribution on $\{|z-z_j|=t\}$. This comparison relies on the following two results:

\begin{fact}[\cite{gn}, Claim 4.7]
There exists a constant $C>0$ such that for all sufficiently small $t>0$,
\[
\frac{1}{N^2}\sum_{j \neq k} \log |z_j - z_k| \leq \Sigma(\mu_{\underline{z}}^t) - \frac{C\log t}{N}.
\]
\end{fact}

\begin{claim}
For any $t>0$ small enough,
$$\frac{1}{N}\log A_\beta(\underline{z}) = 2\sup_{w\in\mathbb{C}}\bigg\{\frac{1}{N}\log|q_{\underline{z}}(w)| - \frac{L^\beta}{2N}|w|^\beta\bigg\} \geq B_{\alpha,\beta}(\mu_{\underline{z}}^t) - 2v_\beta(t,\alpha),$$
where
$$B_{\alpha,\beta}(\mu)\coloneqq2\sup_{w\in \mathbb{C}} \big( U_{\mu}(w) - \frac{|w|^\beta}{\beta\alpha} \big)$$
and
$$v_\beta(t,\alpha)=\begin{cases}2t\alpha^{-1/\beta},&\beta\geq 1,\\\frac{1}{\beta\alpha}t^\beta,&\beta<1.\end{cases}$$
\end{claim}

\begin{proof}
The proof follows the idea of \cite[Claim 4.8]{gn} with some extra effort. Denote $\mu_{\underline{z}}=\tfrac{1}{N}\sum_{k=1}^N \delta_{z_k}$, one has:
\begin{align*}
U_{\mu_{\underline{z}}}(w+te^{i\theta}) - \frac{L^\beta}{2N}|w+te^{i\theta}|^\beta
&= \frac{1}{N}\sum_{j=1}^N \log|w+te^{i\theta}-z_j| - \frac{L^\beta}{2N}|w+te^{i\theta}|^\beta \\
&\leq \frac{1}{2N}\log A_\beta(\underline{z}).
\end{align*}

We make use of the observation from \cite[Example IV.6.2]{st} that $\frac{1}{\alpha}\widehat{m}^{\beta}\big|_{\{|z|\leq\alpha^{1/\beta}\}}$ is the minimizer of $I_{\alpha,\beta}$ over $\mathcal{M}_1(\mathbb{C})$. This allows us to restate the following theorem in our setting:
\begin{theorem}[{\cite[Theorem I.4.1]{st}}]
Let $\mathcal{H}$ denote the set of functions $g(w)$ on $\mathbb{C}$ that are subharmonic on $\mathbb{C}$, harmonic for large $|w|$, and $g(w)-\log |w|$ is bounded above near infinity. Then
$$\sup\bigg\{g(z):\,g\in\mathcal{H},\,g(w)\leq\frac{|w|^\beta}{\beta\alpha}\text{ for any }|w|\leq\alpha^{1/\beta}\bigg\}=\frac{|z|^\beta}{\beta\alpha}.$$
\end{theorem}

Since
$$U_{\mu}(z)-\sup_{|w|\leq\alpha^{1/\beta}} \big( U_{\mu}(w) - \frac{|w|^\beta}{\beta\alpha} \big)$$
satisfies the conditions for $g(z)$ in the theorem above, we obtain that
$$U_{\mu}(z)-\sup_{|w|\leq\alpha^{1/\beta}} \big( U_{\mu}(w) - \frac{|w|^\beta}{\beta\alpha} \big)\leq\frac{|z|^\beta}{\beta\alpha}$$
for all $z\in\mathbb{C}$. Taking the supremum over $z\in\mathbb{C}$ then yields
$$B_{\alpha,\beta}(\mu)=2\sup_{w\in \mathbb{C}} \big( U_{\mu}(w) - \frac{|w|^\beta}{\beta\alpha} \big)=2\sup_{|w|\leq\alpha^{1/\beta}} \big( U_{\mu}(w) - \frac{|w|^\beta}{\beta\alpha} \big).$$

Note that $N=\tfrac{\beta\alpha}{2}r^\beta$ and $L<r$. For sufficiently small $t>0$ and $w\in\mathbb{C}$ with $|w|\leq \alpha^{1/\beta}$, one has:
$$\frac{L^\beta}{2N}|w+te^{i\theta}|^\beta \leq \frac{1}{\beta\alpha}(|w|+t)^\beta\leq\begin{cases}\frac{1}{\beta\alpha}|w|^\beta+2t\alpha^{-1/\beta},&\beta\geq 1,\\ \frac{1}{\beta\alpha}(|w|^\beta+t^\beta),&\beta<1.\end{cases}$$
Therefore,
\begin{align*}
U_{\mu_{\underline{z}}^t}(w)- \frac{|w|^\beta}{\beta\alpha}&=\frac{1}{2\pi}\int_0^{2\pi}U_{\mu_{\underline{z}}}(w+te^{i\theta})d\theta - \frac{|w|^\beta}{\beta\alpha}\\
&\leq\frac{1}{2\pi}\int_0^{2\pi}\bigg(U_{\mu_{\underline{z}}}(w+te^{i\theta})-\frac{L^\beta}{2N}|w+te^{i\theta}|^\beta\bigg)d\theta+v_\beta(t,\alpha)\\
&\leq\frac{1}{2N}\log A_\beta(\underline{z})+v_\beta(t,\alpha).
\end{align*}
The claim follows by taking the supremum over $\{|w|\leq \alpha^{1/\beta}\}$.
\end{proof}

Consequently, we obtain the key comparison:
$$I_\beta^*(\underline{z}) \geq I_{\alpha,\beta}(\mu_{\underline{z}}^t)+\frac{C\log t}{N}-2v_\beta(t,\alpha).$$

Here we only prove part (1) of Proposition~\ref{prop:1.3} and compute the corresponding upper bound, as parts (2) and (3) follow by the same argument. The strict concavity and upper semicontinuity of $\Sigma(\nu)$ (\cite[Proposition 2.2]{hp}) imply that $I_{\alpha,\beta}$ is strictly convex and lower semicontinuous, which guarantees the existence of a unique minimizer. To find the minimum of $I_{\alpha,\beta}$ over $\{\nu\in\mathcal{M}_1(\mathbb{C}):\,\nu(\mathbb{D})\leq p/\alpha\}$ and thereby prove Proposition~\ref{prop:1.3}, it suffices to consider radially symmetric measures. Indeed, for any $\mu\in\mathcal{M}_1(\mathbb{C})$, its radial symmetrization $\mu_{rad}$ satisfies $I_{\alpha,\beta}(\mu_{rad})\leq I_{\alpha,\beta}(\mu)$ owing to the convexity of $I_{\alpha,\beta}$. To identify the minimizer, we use the following key lemma \cite[Lemma 29]{zz}:

\begin{lemma}\label{lem:3.5}
Set $g_{\mu}(z)\coloneqq U_{\mu}(z) - \frac{|z|^{\beta}}{\beta\alpha} -\frac{1}{2}B_{\alpha,\beta}(\mu)$, where $\mu,\nu\in\mathcal{M}_1(\mathbb{C})$ have compact support and finite logarithmic energy. If $g_\nu(z)\leq g_\mu(z)$ for $\mu$-a.e.~$z$, then $I_{\alpha,\beta}(\mu)\leq I_{\alpha,\beta}(\nu)$.
\end{lemma}

To minimize $I_{\alpha,\beta}$ over $\{\nu\in\mathcal{M}_1(\mathbb{C}):\,\nu(\mathbb{D})\leq p/\alpha\}$, we apply the method of undetermined coefficients to construct a measure $\mu\in\mathcal{M}_1(\mathbb{C})$ such that for every $\nu$ in this class with compact support and finite logarithmic energy, we have $g_\nu(z)\leq g_\mu(z)$ for $\mu$-a.e.~$z$. Claim~\ref{lem:3.5} then directly implies that $\mu$ is the minimizer.

Recall that $\frac{1}{\alpha}\widehat{m}^{\beta}\big|_{\{|z|\leq\alpha^{1/\beta}\}}$ is the minimizer of $I_{\alpha,\beta}$ over $\mathcal{M}_1(\mathbb{C})$, since we must minimize over the constrained set $\{\nu\in\mathcal{M}_1(\mathbb{C}):\,\nu(\mathbb{D})\leq p/\alpha\}$, we need to adjust the measure $\frac{1}{\alpha}\widehat{m}^{\beta}\big|_{\{|z|\leq\alpha^{1/\beta}\}}$.

First consider the case $0<p<1$. In the case $\beta=2$, the minimizer of $I_{\alpha,2}$ obtained in \cite{gn} takes the explicit form
$$\mu_{\alpha,p}^2=\frac{q-p}{\alpha}m_{\mathbb{S}^1}+\frac{1}{\alpha}\bigg(\widehat{m}^2\big|_{\{|z|\leq p^{1/2}\}}+\widehat{m}^2\big|_{\{q^{1/2}\leq |z|\leq \alpha^{1/2}\}}\bigg).$$
Guided by this expression, we assume that for general $\beta>0$, the minimizer admits an analogous representation.
$$
\mu_{\alpha,p}^{\beta}=Am_{\mathbb{S}^1}+B\widehat{m}^{\beta}\big|_{\{|z|\leq C\}}+D\widehat{m}^{\beta}\big|_{\{E\leq |z|\leq \alpha^{1/\beta}\}}
$$
with $A,B,D\geq 0$ and $0<C<1<E<\alpha^{1/\beta}$. A direct computation shows:
$$
U_{\mu_{\alpha,p}^\beta}(z) = \begin{cases}
B(C^\beta\log C-\frac{C^\beta}{\beta}+\frac{|z|^\beta}{\beta})+D(\frac{\alpha\log\alpha}{\beta}-E^\beta\log E-\frac{\alpha}{\beta}+\frac{E^\beta}{\beta}), & |z|\leq C, \\
BC^\beta \log|z|+D(\frac{\alpha\log\alpha}{\beta}-E^\beta\log E-\frac{\alpha}{\beta}+\frac{E^\beta}{\beta}),& C< |z|\leq 1, \\
BC^\beta \log|z|+D(\frac{\alpha\log\alpha}{\beta}-E^\beta\log E-\frac{\alpha}{\beta}+\frac{E^\beta}{\beta})+A\log|z|, & 1< |z| < E, \\
BC^\beta \log|z|+D(\frac{\alpha\log\alpha}{\beta}-E^\beta\log |z|-\frac{\alpha}{\beta}+\frac{|z|^\beta}{\beta})+A\log|z|, & E\leq|z|\leq\alpha^{1/\beta},  \\
BC^\beta \log|z|+D(\alpha\log|z|-E^\beta\log|z|)+A\log|z|, & |z|>\alpha^{1/\beta}.
\end{cases}
$$

To meet the desired condition, we choose parameters so that $U_{\mu_{\alpha,p}^\beta}(z)-\frac{|z|^\beta}{\beta\alpha}=\sup_{w\in \mathbb{C}} \big( U_{\mu}(w) - \frac{|w|^\beta}{\beta\alpha} \big)$ holds on part of the support of $\mu_{\alpha,p}^\beta$. Hence we require the parameters to satisfy:
\begin{gather}
\label{eq:3.6}
    B=D=\frac{1}{\alpha}\\
    BC^\beta-DE^\beta+A=0\\
    B(C^\beta\log C-\frac{C^\beta}{\beta})+D(\frac{\alpha\log\alpha}{\beta}-E^\beta\log E-\frac{\alpha}{\beta}+\frac{E^\beta}{\beta})=D(\frac{\alpha\log\alpha}{\beta}-\frac{\alpha}{\beta})
\end{gather}
so that $U_{\mu_{\alpha,p}^\beta}(z)-\frac{|z|^\beta}{\beta\alpha}$ is constant (which we temporarily assume equals the supremum $B_{\alpha,\beta}(\mu)$) on $\{|z|\leq C\}\cup\{E\leq|z|\leq\alpha^{1/\beta}\}$. The condition $\mu_{\alpha,p}^\beta(\mathbb{D})\leq\frac{p}{\alpha}$ implies
\begin{align}
\label{eq:3.9}
    DC^\beta\leq\frac{p}{\alpha}.
\end{align}

For any radially symmetric $\nu$ with compact support satisfying $\nu(\mathbb{D})\leq p/\alpha$, Jensen's formula yields
\begin{align}
g_{\nu}(z)&=g_{\nu}(1)\nonumber\\
&\leq U_{\nu}(1)-\frac{1}{\beta\alpha}-U_{\nu}(p^{1/\beta})+\frac{p}{\beta\alpha} \nonumber\\
&\leq\int_{p^{1/\beta}}^{1}\frac{\nu(\overline{D(0,t)})}{t}dt+\frac{p-1}{\beta\alpha}\nonumber \\
&\leq\frac{1}{\beta\alpha}(p-1-p\log p)
\end{align}
for any $z\in\mathbb{C}$ with $|z|=1$. To obtain an upper bound for all $g_{\nu}(z)$ on $\{|z|=1\}$, we assume that $g_{\mu_{\alpha,p}^\beta}(z)=U_{\mu_{\alpha,p}^\beta}(z) - \frac{|z|^{\beta}}{\beta\alpha} -\frac{1}{2}B_{\alpha,\beta}(\mu)\geq\frac{1}{\beta\alpha}(p-1-p\log p)$ on $\{|z|=1\}$, i.e.,
\begin{align}
\label{eq:3.11}
    D\Big(\frac{\alpha\log\alpha}{\beta}-E^\beta\log E-\frac{\alpha}{\beta}&+\frac{E^\beta}{\beta}\Big)-\frac{1}{\beta\alpha}-\Big[B\Big(C^\beta\log C-\frac{C^\beta}{\beta}\Big)\nonumber\\
    &\qquad+D\Big(\frac{\alpha\log\alpha}{\beta}-E^\beta\log E-\frac{\alpha}{\beta}+\frac{E^\beta}{\beta}\Big)\Big]\geq\frac{1}{\beta\alpha}(p-1-p\log p).
\end{align}

Combining \eqref{eq:3.6}--\eqref{eq:3.9} and \eqref{eq:3.11}, along with the definition $q(\log q-1)=p(\log p-1)$ (where $q\neq p$), we obtain
$$
B=D=\frac{1}{\alpha},\qquad C=p^{1/\beta},\qquad E=q^{1/\beta},\qquad A=\frac{q-p}{\alpha}.
$$
We therefore conjecture that the minimizer is
$$
\mu_{\alpha,p}^{\beta}=\frac{q-p}{\alpha}m_{\mathbb{S}^1}+\frac{1}{\alpha}\bigg(\widehat{m}^{\beta}\big|_{\{|z|\leq p^{1/\beta}\}}+\widehat{m}^{\beta}\big|_{\{q^{1/\beta}\leq |z|\leq \alpha^{1/\beta}\}}\bigg).
$$

We now verify this. Since
$$
U_{\mu_{\alpha,p}^\beta}(z) = \begin{cases}
\frac{\log\alpha-1}{\beta}+\frac{|z|^\beta}{\beta\alpha}, & |z|\leq p^{1/\beta}, \\
\frac{p\log|z|}{\alpha}+\frac{\log\alpha-1}{\beta}-\frac{p\log p-p}{\beta\alpha} ,& p^{1/\beta}< |z|\leq 1, \\
\frac{q\log|z|}{\alpha}+\frac{\log\alpha-1}{\beta}-\frac{q\log q-q}{\beta\alpha}, & 1< |z| < q^{1/\beta}, \\
\frac{\log\alpha-1}{\beta}+\frac{|z|^\beta}{\beta\alpha}, & q^{1/\beta}\leq |z| \leq \alpha^{1/\beta},  \\
\log|z|, & |z|>\alpha^{1/\beta},
\end{cases}
$$
one has $$B_{\alpha,\beta}(\mu_{\alpha,p}^\beta)=2\sup_{w\in \mathbb{C}} \big( U_{\mu}(w) - \frac{|w|^\beta}{\beta\alpha} \big)=\frac{2}{\beta}(\log\alpha-1).$$Hence $$g_{\mu_{\alpha,p}^\beta}(z) = 0$$ on $\{|z| \leq p^{1/\beta}\}\cup\{q^{1/\beta}\leq |z| \leq \alpha^{1/\beta}\}$ and $$g_{\mu_{\alpha,p}^\beta}(z) =\frac{1}{\beta\alpha}(p-1-p\log p)$$ on $\{|z|=1\}$. Noting that $g_{\nu}(z)\leq 0$ always holds, by Claim~\ref{lem:3.5} we conclude that $\mu_{\alpha,p}^\beta$ is the unique minimizer of $I_{\alpha,\beta}(\mu)$ on $\{\mu\in\mathcal{M}_1(\mathbb{C}):\,\mu(\mathbb{D})\leq p/\alpha\}$.

For the case $p=0$, the same approach gives the minimizer
$$
\mu_{\alpha,0}^{\beta}=\frac{e}{\alpha}m_{\mathbb{S}^1}+\frac{1}{\alpha}\widehat{m}^{\beta}\big|_{\{e^{1/\beta}\leq |z|\leq\alpha^{1/\beta}\}}.$$

A direct computation shows that
$$I_{\alpha,\beta}(\mu_{\alpha,p}^\beta)=\frac{1}{\beta}(\log\alpha-\frac{3}{2})+\frac{2}{\beta\alpha^2}Z_p.$$
Since $L = \tfrac{r-K_0}{1+t} = r + o(1)$, $N=\tfrac{\beta\alpha}{2}r^\beta$ and $n_{P_{N,L}}(1+t) = n_{P_N}(r-K_0) \leq n_{F_\beta}(r)$, we obtain:
\begin{align*}
\left\{n_{F_\beta}(r) \leq\lfloor\frac{\beta p}{2}r^\beta \rfloor \right\} &\subset \left\{n_{P_{N,L}}(1+t) \leq\frac{\beta p}{2}r^\beta=\frac{pN}{\alpha}\right\}\subset \big\{\mu_{\underline{z}}^t(\mathbb{D}) \leq p/\alpha\big\}.
\end{align*}

Observing that $\{\mu_{\underline{z}}^t:\,\mu_{\underline{z}}^t(\mathbb{D})\leq p/\alpha\}\subset\{\mu:\,\mu(\mathbb{D})\leq p/\alpha\}$, one has:
\begin{align*}
\mathbb{P}\bigg[\Big\{n_{F_\beta}(r)\leq\lfloor\frac{\beta p}{2}r^\beta \rfloor\Big\}\cap E_{reg}\bigg] &\leq\mathbb{P}\bigg[\big\{\mu_{\underline{z}}^t(\mathbb{D}) \leq p/\alpha\big\}\cap E_{reg}\bigg]\\
&\leq\exp\bigg(\frac{\beta\alpha^2 r^{2\beta}}{4}\big(\log\alpha-\frac{3}{2}-\beta\inf_{\mu_{\underline{z}}^t(\mathbb{D})\leq p/\alpha}I_{\alpha,\beta}(\mu_{\underline{z}}^t)\big)+O(r^\beta\log^2 r)\bigg)\\
&\leq\exp\bigg(\frac{\beta\alpha^2 r^{2\beta}}{4}\big(\log\alpha-\frac{3}{2}-\beta\inf_{\mu(\mathbb{D})\leq p/\alpha}I_{\alpha,\beta}(\mu)\big)+O(r^\beta\log^2 r)\bigg)\\
&=\exp\big(-\frac{\beta Z_p}{2}r^{2\beta}+O(r^\beta\log^2 r)\big).
\end{align*}

From (i) of Lemma~\ref{lem:2.8}, we can choose $A$ and $H$ large enough such that $\mathbb{P}\big[E_{reg}^c\big]\leq \exp(-\beta Z_p r^{2\beta})$. Hence
\begin{align*}
\mathbb{P}\bigg[n_{F_\beta}(r) \leq\lfloor\frac{\beta p}{2}r^\beta\rfloor\bigg]&\leq\mathbb{P}\bigg[\Big\{n_{F_\beta}(r) \leq\frac{\beta p}{2}r^\beta\Big\}\cap E_{reg}\bigg]+\mathbb{P}\big[E_{reg}^c\big]\\
&\leq\exp\bigg(-\frac{\beta Z_p}{2}r^{2\beta}+O\big(r^\beta\log^2 r\big)\bigg).
\end{align*}

This completes the proof of the upper bounds of the hole probabilities for $p\in[0,1)$ (part (1) of Proposition~\ref{prop:1.3}). The upper bounds of the hole probabilities for $p>1$ (parts (2) and (3)) follow by the same approach.

\section{The Lower Bounds of Hole Probabilities}
\label{sect:4}
In this section, we complete the proof of Proposition~\ref{prop:1.4} by establishing the matching lower bounds for the hole probabilities $\mathbb{P}[n_{F_\beta}(r) =k_0]$, where $k_0=\lfloor\tfrac{\beta p}{2}r^\beta\rfloor$. Following \cite{gn}, we apply Rouché's theorem by ensuring a dominant monomial term in the series expansion of $F_\beta$ on $\{|z| \leq r\}$. By Rouch\'e's Theorem, one has
$$\big\{|\xi_{k_0}|b_{k_0}>\sum_{k\neq k_0}|\xi_k|b_k\big\}\subset\big\{n_{F_\beta}(r) =k_0\big\},$$
so it suffices to find a lower bound for $\mathbb{P}[|\xi_{k_0}|b_{k_0}>\sum_{k\neq k_0}|\xi_k|b_k]$.

Define
$$b_k = \tfrac{r^k}{\sqrt{\Gamma\big(\frac{2}{\beta}(k+1)\big)}} \quad \text{and} \quad A_k = \frac{b_{k_0}}{b_k} = \frac{r^{k_0}}{r^k}\sqrt{\tfrac{\Gamma\left(\frac{2}{\beta}(k+1)\right)}{\Gamma\left(\frac{2}{\beta}(k_0+1)\right)}}.$$
There exists a constant $C_1>0$ such that for all $k\in\mathbb{N}_+$,
$$ -C_1 \log (k_0+1) \leq \log A_k +\frac{k}{\beta}\log\frac{\beta er^\beta}{2k}-\frac{pr^\beta}{2}\log\frac{e}{p} \leq C_1 \log (k+1).$$

For any $p\in(0,1)$, consider the set
$$ M_p \coloneqq \left\{ k \in \mathbb{N} : p\leq\frac{2k}{\beta r^\beta} \leq q,\,k\neq k_0 \right\}$$
and the event
$$ E_p \coloneqq\left\{ |\xi_k| \leq \frac{A_k}{4\beta (k+1)^{2C_1}r^\beta} \text{ for all } k \in M_p \right\}.$$
On $\{ |\xi_{k_0}| \geq 1 \} \cap E_p$, since $|M_p| < 2\beta r^\beta$, we obtain:
$$ \sum_{k \in M_p} |\xi_k| b_k \leq \sum_{k \in M_p} \frac{A_k}{4\beta (k+1)^{2C_1} r^\beta} b_k \leq \frac{1}{2} |\xi_{k_0}| b_{k_0}. $$

Note that$$\log A_k\leq C_1\log (k+1)+\frac{pr^\beta}{2}\log\frac{e}{p}-\frac{k}{\beta}\log\frac{\beta er^\beta}{2k}\leq C_1\log (k+1)$$ for all $k\in M_p$. There exists $C_2,C_2' > 0$ such that for $k \in M_p$ and sufficiently large $r\gg 1$,
$$\mathbb{P}\left[ |\xi_k| \leq \frac{A_k}{4\beta(k+1)^{2C_1} r^\beta} \right] \geq \frac{C_2 A_k^2}{(k+1)^{4C_1}r^{2\beta}} \geq \exp\bigg(pr^\beta\log\frac{e}{p}- \frac{2k}{\beta}\log\frac{\beta er^\beta}{2k}- C_2'\log r\bigg).$$
Therefore,
\begin{align}
\label{eq:4.1}
\mathbb{P}[E_p] &\geq \exp\bigg( \sum_{k \in M_p} \big(pr^\beta\log\tfrac{e}{p}- \tfrac{2k}{\beta}\log\tfrac{\beta er^\beta}{2k}\big) + O(r^\beta\log r) \bigg) \nonumber\\
&=\exp\bigg(\frac{\beta r^{2\beta}}{2}\sum_{k\in M_p}\big(\frac{1}{\frac{1}{2}\beta r^\beta}\frac{k}{\frac{1}{2}\beta r^\beta}\log\frac{k}{\frac{1}{2}\beta r^\beta}\big)-\frac{\beta(q+p)(q-p)}{4}r^{2\beta}\nonumber\\
&\qquad+\frac{\beta p(q-p)(1-\log p)}{2}r^{2\beta}+O(r^\beta\log r)\bigg) \nonumber\\
&=\exp\big(\frac{\beta r^{2\beta}}{2}(\int_p^q x\log xdx+q(q-\log q)-p^2(1-\log p)-(q^2-p^2))+O(r^\beta\log r)\big)\nonumber\\
&=\exp\big(-\frac{\beta Z_p}{2}r^{2\beta}+O(r^\beta\log r)\big).
\end{align}

We now analyze the tails as follows:

1. \textit{Control of the high-degree tail ($k > N$)}: We apply Lemma~\ref{lem:2.1} with $B = 1$ and choose $\alpha_0=\alpha_0(\beta)>0$ such that $\alpha_0>\max\{4,4^\beta,e\}$. Since $N=\frac{\beta\alpha_0}{2}r^\beta$, outside an event of probability at most $\exp(-C_0 r^{3\beta})$, and on the event $\{|\xi_{k_0}|\geq 1\}$, one has
$$ |T_N(z)| \leq \exp\left(\frac{\alpha_0}{2}r^\beta\log\frac{4}{\alpha_0}\right)=\exp(-c_{\alpha_0}r^\beta) < \frac{1}{4}|\xi_{k_0}|b_{k_0} \quad \text{for } |z| \leq r $$
when $r$ is sufficiently large.

2. \textit{Control of the remaining terms ($k \leq N$, $k \notin M_p$)}: For indices $k\leq N$ with $k\neq k_0$ but not in $M_p$ (denoted by $M_p'$), we have $|M_p'| \leq \lfloor\beta 4^{\beta+1}r^\beta\rfloor$. For any $k\in M_p'$, there exists a positive constant $C_3$, such that
$$\log A_k \geq \frac{pr^\beta}{2}\log\frac{e}{p}- \frac{k}{\beta}\log\frac{\beta er^\beta}{2k}-C_1\log (k_0+1)\geq -C_3\log(k+1).$$

Set
$$ E_p'\coloneqq \bigg\{ |\xi_k| \leq \frac{1}{\beta 4^{\beta+2} r^\beta(k+1)^{C_3}} \text{ for all } k \in M_p' \bigg\},$$
then $$ \sum_{k \in M_p'} |\xi_k| b_k \leq\sum_{k \in M_p'} \frac{b_k A_k}{\beta 4^{\beta+2} r^\beta } \leq\frac{1}{4} |\xi_{k_0}| b_{k_0}$$on the event $\{ |\xi_{k_0}| \geq 1 \} \cap E_p'$.

Moreover, there exists a constant $C_4>0$, such that
\begin{align}
\label{eq:4.2}
    \mathbb{P}[E_p'] \geq \prod_{k=0}^{\lfloor\beta 4^{\beta+1} r^\beta\rfloor} \frac{C_4}{r^{2\beta}(k+1)^{2C_3}}= \exp\big(O(r^\beta\log r)\big).
\end{align}

Combining these results, we obtain:
\begin{align*}
\mathbb{P}\big[n_{F_\beta}(r) = \lfloor \frac{\beta p}{2}r^\beta\rfloor\big] &\geq \mathbb{P}\bigg[ |\xi_{k_0}|b_{k_0} > \sum_{k\neq k_0} |\xi_k| b_k \bigg] \\
&\geq c\mathbb{P}[E_p]\,\mathbb{P}[E_p'] \\
\text{[by \eqref{eq:4.1} and \eqref{eq:4.2}]  }&=\exp\big(-\frac{\beta Z_p}{2}r^{2\beta}+O(r^\beta\log r)\big).
\end{align*}

The case $p=0$ can be handled by an identical argument, where the dominant term is the constant coefficient $\xi_0$, and the sets $M_0$ are defined accordingly. As for $p\in(1,e)$, set $$M_p = \left\{ k \in \mathbb{N} : q\leq\frac{2k}{\beta r^\beta}\leq p,\,k\neq k_0 \right\},$$ and for $p\geq e$, set $$M_p = \left\{ k \in \mathbb{N} : \frac{2k}{\beta r^\beta}\leq p,\,k\neq k_0 \right\}.$$

For $p \in (1,e)$ and $p \ge e$, the argument proceeds analogously, with the set $M_p$ redefined as above to capture the relevant range of indices where the coefficients need to be controlled. The estimates for $A_k$ and the subsequent probability bounds follow the same pattern. Together with the results from Section~\ref{sect:3}, we establish Proposition~\ref{prop:1.4}.

\section{Proof of the Main Results}
\label{sect:5}
 This section proves Lemma~\ref{lem:1.5}, Theorem~\ref{thm:1.1}, and Corollary~\ref{cor:1.2}.
\begin{proofof}{Lemma~\ref{lem:1.5}}
The proof follows the strategy of Section 7 in \cite{gn}, adapted to our family of weights. For brevity, we present the details for the case $p<1$; the other cases are analogous. Denote $\mathfrak{D}(\phi)=\int_{\mathbb{C}}(\phi_x^2+\phi_y^2)dm(z)$ as the Dirichlet energy of $\phi$, where $\phi \in C_c^{\infty}(\mathbb{C})$ is a test function with compact support.

Since $N=\tfrac{\beta\alpha}{2}r^\beta$ and $\phi$ has compact support, for sufficiently large $\alpha\gg 1$, one has
$$r^\beta\int_{\mathbb{C}}\phi\, d\mu_p^\beta=\frac{\beta\alpha r^\beta}{2} \int_{\mathbb{C}}\phi\, d\mu_{\alpha,p}^\beta=N\int_{\mathbb{C}}\phi\, d\mu_{\alpha,p}^\beta.$$
Note that $L = \tfrac{r-K_0}{1+t}$, which implies
$$\frac{r}{L}=(1+t)\left(1+\frac{K_0}{r-K_0}\right)=1+O(r^{-s})+O(r^{2\beta-s-1})$$as $r\rightarrow \infty$. Hence, since $\text{supp}\,\phi\subset D(0,B)$, we can obtain
\begin{align*}
\sum_{P_{N,L}(z)=0}\phi(z) &= n_{P_{N,L}}\big(\phi;\frac{L}{r}\big)+O(BNr^{-s})+O(BNr^{2\beta-s-1})\\
&=n_{F_\beta}(\phi;r)+O(BNr^{-s})+O(BNr^{2\beta-s-1})+O\big(M_0\omega(\phi;2(2B)^{2\beta} r^{2\beta-s-1})\big).
\end{align*}

Since $\phi$ is smooth with compact support, there exists a constant $\widehat{C}_\phi>0$ such that $|\omega(\phi;t)|\leq \widehat{C}_\phi t$. The constant $s$ can be chosen such that
$$
\Big|n_{F_\beta}(\phi;r)-\sum_{P_{N,L}(z)=0}\phi(z)\Big|<\frac{1}{r}.\qquad(r\gg1)
$$

For any $\kappa>0$, set $\kappa'=\kappa-|n_{F_\beta}(\phi;r)-\sum_{P_{N,L}(z)=0}\phi(z)|$. Then
$$\bigg|\frac{1}{N}\sum_{P_{N,L}(z)=0}\phi(z)-\int_{\mathbb{C}}\phi\, d\mu_{\alpha,p}^\beta\bigg| \geq \frac{\kappa}{N} - \bigg|\frac{1}{N}n_{F_\beta}(\phi;r)-\frac{1}{N}\sum_{P_{N,L}(z)=0}\phi(z)\bigg|= \frac{\kappa'}{N}>\frac{1}{N}\bigg(\kappa-\frac{1}{r}\bigg)>0$$
on $L_\beta(p,\phi,\kappa;r)\cap E_{reg}$ when $r$ is large enough.

To relate the deviation in linear statistics to the energy functional, we invoke the following estimate, which is a restatement of \cite[Claim 7.4]{gn} adapted to our setting.

\begin{claim}
\label{clm:5.1}
Consider the set
$$\mathcal{L}_{\phi,\tau_{\alpha,\beta},\lambda}\coloneqq\left\{\nu\in\mathcal{M}_1(\mathbb{C}):\left|\int_{\mathbb{C}}\phi(w)\,d\nu(w)-\tau_{\alpha,\beta}\right|\geq\lambda\right\},$$
where $\tau_{\alpha,\beta}=\int_{\mathbb{C}}\phi\, d\mu_{\alpha,p}^\beta$. For any $\nu\in\mathcal{L}_{\phi,\tau_{\alpha,\beta},\lambda}$ with $\nu(\mathbb{D})\leq p/\alpha$, one has
\begin{align}
\label{eq:5.1}
    I_{\alpha,\beta}(\nu)-I_{\alpha,\beta}(\mu_{\alpha,p}^\beta)\geq\frac{2\pi}{\mathfrak D(\phi)}\lambda^2.
\end{align}
\end{claim}

Denote
$$L_{\phi,\tau,\eta}^N(t)\coloneqq\bigg\{\underline{z}:\Big|\frac{1}{N}\sum_{j=1}^N\phi(z_j)-\tau\Big|\geq\eta+\omega(\phi;t)\bigg\}.$$
For each $z_j$, we can obtain the bound $|\phi(z_j)-\int_{\mathbb{C}}\phi\,dm_{|z-z_j|=t}|\leq \omega(\phi;t)$. Choose $\eta=(\kappa'-1/r)/N$, then for any $\underline{z}\in L_{\phi,\tau_{\alpha,\beta},\eta}^N(t)$, one has $\mu_{\underline{z}}^t\in \mathcal{L}_{\phi,\tau_{\alpha,\beta},\eta}$. Hence
$$\bigg\{\mu_{\underline{z}}^t:\underline{z}\in\big\{n_{P_{N,L}}(1+t)\leq\frac{\beta p}{2}r^\beta\big\}\cap L_{\phi,\tau_{\alpha,\beta},\eta}^N(t)\bigg\}\subset \big\{\mu:\,\mu(\mathbb{D})\leq\frac{p}{\alpha}\big\}\cap\mathcal{L}_{\phi,\tau_{\alpha,\beta},\eta}.$$
Applying Claim~\ref{clm:5.1} with $\lambda=\eta$, we obtain:
\begin{align*}
\mathbb{P}\bigg[\big\{n_{P_{N,L}}(1+t)\leq\frac{\beta p}{2}r^\beta\big\}&\cap L_{\phi,\tau_{\alpha,\beta},\eta}^N(t)\cap E_{reg}\bigg]\\
&\leq\exp\bigg(\frac{\beta\alpha^2 r^{2\beta}}{4}\big(\log\alpha-\frac{3}{2}-\beta\inf\big\{I_{\alpha,\beta}(\nu):\nu\in\mathcal{L}_{\phi,\tau_{\alpha,\beta},\eta},\\
&\qquad\nu(\mathbb{D})\leq \frac{p}{\alpha}\big\}\big)+O(r^\beta\log^2 r)\bigg)\\
\text{[by \eqref{eq:5.1}]\qquad}&\leq\mathbb{P}\big[n_{F_\beta}(r) \leq\frac{\beta p}{2}r^\beta\big]\exp\big(-\frac{2\pi}{\mathfrak{D}(\phi)}(\kappa'-1/r)^2+O(r^\beta\log^2 r)\big).
\end{align*}

By (1) of Lemma~\ref{lem:2.8}, given any $M>0$, we can ensure that $\mathbb{P}[E_{reg}^c]\leq \exp(-Mr^{2\beta})$. Consequently, for every $\kappa\in(0,Sr^\beta)$, the following holds:
\begin{align*}
\mathbb{P}\bigg[\big\{n_{F_\beta}(r)\leq\frac{\beta p}{2}r^\beta\big\}\cap L_\beta(p,\phi,\kappa;r)\bigg]&\leq\mathbb{P}\big[E_{reg}^c\big]+\mathbb{P}\big[n_{F_\beta}(r) \leq\frac{\beta p}{2}r^\beta\big]\exp\big(-\frac{2\pi}{\mathfrak{D}(\phi)}(\kappa'-1/r)^2+O(r^\beta\log^2 r)\big)\\
&\leq \mathbb{P}\big[n_{F_\beta}(r) \leq\frac{\beta p}{2}r^\beta\big]\exp\big(-\frac{2\pi}{\mathfrak{D}(\phi)}\kappa^2+C_\phi r^\beta\log^2 r\big).
\end{align*}
Thus, the proof is complete. 
\end{proofof}
We now proceed to prove our main result.
\begin{proofof}{Theorem~\ref{thm:1.1}}
We begin with the following fact:
\begin{fact}[{\cite[Claim 7.8]{gn}}]
For any $a\in\mathbb{R}$, $T\geq 0$, and a random variable $X$ with finite mean, one has
$$\mathbb{E}|X-a|\leq T+\int_0^\infty\mathbb{P}[|X-a|>u+T]\,du.$$
\end{fact}

Taking $X=X_{\phi,\beta}(r)\coloneqq n_{F_\beta}(\phi;r)|_{n_{F_\beta}(r)=0}$, $a=r^{\beta}\int_{\mathbb{C}}\phi(z)\,d\mu_{0}^{\beta}(z)$ and $T=C_\phi'r^{\beta/2}\log r$, where $$C_\phi'>\sqrt{\frac{\mathfrak{D}(\phi)}{2\pi}C_\phi}$$ is a constant. Then
\begin{align*}
    \Big|\mathbb{E}[X_{\phi,\beta}(r)]-r^{\beta}\int_{\mathbb{C}}\phi(z)\,d\mu_{0}^{\beta}(z)\Big|&\leq\mathbb{E}\bigg[\Big|X_{\phi,\beta}(r)-r^{\beta}\int_{\mathbb{C}}\phi(z)\,d\mu_{0}^{\beta}(z)\Big|\bigg]\\
    &\leq C_\phi'r^{\beta/2}\log r+\int_0^\infty\mathbb{P}\bigg[\big|X_{\phi,\beta}(r)-r^{\beta}\int_{\mathbb{C}}\phi(z)\,d\mu_{0}^{\beta}(z)\big|\\
    &\qquad>u+C_\phi'r^{\beta/2}\log r\bigg]\,du\\
    \text{[by Lemma~\ref{lem:1.5}]\qquad}&=O(r^{\beta/2}\log r).
\end{align*}
This completes the proof of Theorem~\ref{thm:1.1}.
\end{proofof}

\begin{proofof}{Corollary~\ref{cor:1.2}}
By the definition of $X_{\phi,\beta}(r)$, it suffices to show that for any continuous and compactly supported function $f$ on $\mathbb{C}$,
$$\frac{X_{f,\beta}(r)}{r^\beta}\longrightarrow \int_{\mathbb{C}}f(z)\,d\mu_{0}^{\beta}(z)
$$ 
in distribution as $r\rightarrow \infty$. Assume $\text{supp }f\subset D(0,B)$. For any $\delta>0$, we can choose a smooth function $g$ with compact support in $D(0,B+1)$ such that $|f(z)-g(z)|\leq \delta$ for all $z\in\mathbb{C}$. Then
\begin{align*}
    |X_{f,\beta}(r)-X_{g,\beta}(r)|&\leq\sum_{F_\beta(z)=0}\Big|f\Big(\frac{z}{r}\Big)-g\Big(\frac{z}{r}\Big)\Big|\leq  n_{F_\beta}\big(r(B+1)\big)\delta,\\
    \bigg|\int_{\mathbb{C}}f(z)\,d\mu_{0}^{\beta}(z)-\int_{\mathbb{C}}g(z)\,d\mu_{0}^{\beta}(z)\bigg|&\leq\int_{D(0,B+1)}|f(z)-g(z)|\,d\mu_{0}^{\beta}(z)\leq \frac{\beta}{2}\delta(B+1)^\beta.
\end{align*}
Therefore,
\begin{align*}
   \bigg|\frac{X_{f,\beta}(r)}{r^\beta}- \int_{\mathbb{C}}f(z)\,d\mu_{0}^{\beta}(z)\bigg|&\leq\bigg|\frac{X_{f,\beta}(r)}{r^\beta}- \frac{X_{g,\beta}(r)}{r^\beta}\bigg|+\bigg|\frac{X_{g,\beta}(r)}{r^\beta}- \int_{\mathbb{C}}g(z)\,d\mu_{0}^{\beta}(z)\bigg|\\
   &\qquad+\bigg|\int_{\mathbb{C}}g(z)\,d\mu_{0}^{\beta}(z)- \int_{\mathbb{C}}f(z)\,d\mu_{0}^{\beta}(z)\bigg|\\
   &\leq\bigg|\frac{X_{f,\beta}(r)}{r^\beta}- \frac{X_{g,\beta}(r)}{r^\beta}\bigg|+\bigg|\frac{X_{g,\beta}(r)}{r^\beta}- \int_{\mathbb{C}}g(z)\,d\mu_{0}^{\beta}(z)\bigg|+\frac{\beta}{2}\delta(B+1)^\beta.
\end{align*}
For any $\epsilon\in(0,1)$, choose $\delta=\frac{2\epsilon^2}{\beta(B+1)^\beta}$, then one has:
\begin{align*}
    \mathbb{P}\bigg[\bigg|\frac{X_{f,\beta}(r)}{r^\beta}- \int_{\mathbb{C}}f(z)\,d\mu_{0}^{\beta}(z)\bigg|>\epsilon\bigg]&\leq\mathbb{P}\bigg[\bigg|\frac{X_{f,\beta}(r)}{r^\beta}- \frac{X_{g,\beta}(r)}{r^\beta}\bigg|>\epsilon\bigg]\\ &\qquad+\mathbb{P}\bigg[\bigg|\frac{X_{g,\beta}(r)}{r^\beta}- \int_{\mathbb{C}}g(z)\,d\mu_{0}^{\beta}(z)\bigg|>\epsilon-\frac{\beta}{2}\delta(B+1)^\beta\bigg]\\
    &\leq\mathbb{P}\bigg[|n_{F_\beta}\big(r(B+1)\big)|>\frac{\beta}{2\epsilon}\big(r(B+1)\big)^\beta\,\big|\,n_{F_\beta}(r)=0\bigg]\\
    &\qquad+\mathbb{P}\bigg[\bigg|\frac{X_{g,\beta}(r)}{r^\beta}- \int_{\mathbb{C}}g(z)\,d\mu_{0}^{\beta}(z)\bigg|>\epsilon-\epsilon^2\bigg].
\end{align*}

By the proof of Theorem~\ref{thm:1.1}, as $r\rightarrow \infty$, $$\frac{X_{g,\beta}(r)}{r^\beta}\longrightarrow \int_{\mathbb{C}}g(z)\,d\mu_{0}^{\beta}(z)$$ in $L^1$-norm, hence also in probability. Together with the bound
$$\mathbb{P}\bigg[\big|n_{F_\beta}\big(r(B+1)\big)\big|>\frac{\beta}{2\epsilon}\big(r(B+1)\big)^\beta\,\big|\,n_{F_\beta}(r)=0\bigg]\leq\exp(-C_{B,\epsilon}r^{2\beta}),$$
we conclude that
$$\frac{X_{f,\beta}(r)}{r^\beta}\longrightarrow \int_{\mathbb{C}}f(z)\,d\mu_{0}^{\beta}(z)$$
in probability, and therefore in distribution. This finishes the proof.
\end{proofof}

\section*{Acknowledgements}
The author would like to express sincere gratitude to the author's advisor, Professor Song-Yan Xie, for invaluable guidance and support throughout this research. The author also thanks Professor Hao Wu of Nanjing University and Bin Guo of Academy of Mathematics and Systems Science for helpful comments.

\medskip

\section*{Funding}

This work was partially supported by the National Natural Science Foundation of China under Grant No. 12471081.

\bibliographystyle{unsrt} 
\bibliography{abc}
\end{document}